\documentclass{amsart}
\usepackage{amsmath,amssymb}
\usepackage[left=1.5in,right=1.5in,bottom=1in,top=1in]{geometry}
\usepackage[all]{xy}
\usepackage{tikz}
\usetikzlibrary{cd}
\usepackage[color=cyan]{todonotes}
\usepackage{multicol}
\usepackage{xcolor}

\makeatletter
\providecommand\@dotsep{5}
\renewcommand{\listoftodos}[1][\@todonotes@todolistname]{%
  \@starttoc{tdo}{#1}}
\makeatother

\newcommand{\ord}{\mathrm{ord}}
\newcommand{\A}{\mathbb{A}}

\newcommand{\bG}{ \boldsymbol{G}    }
\newcommand{\biota}{ \boldsymbol{\iota}    }
\newcommand{\blambda}{ \boldsymbol{\lambda}    }
\newcommand{\Q}{\mathbb{Q}}
\newcommand{\Z}{\mathbb{Z}}
\newcommand{\C}{\mathbb{C}}

\newcommand{\Hom}{\mathrm{Hom}}
\newcommand{\End}{\mathrm{End}}

\newcommand{\GU}{\mathrm{GU}}

\newcommand{\R}{\mathbb{R}}

\newcommand{\G}{\mathbb{G}}

\title{Supersingular Loci of Some Shimura Varieties for  CM Unitary Groups}
\author{Maria Fox}
\address{Department of Mathematics \\ University of Oregon }
\email{mariafox@uoregon.edu}

\begin{document}

\newtheorem*{theorem*}{Theorem}
\newtheorem*{theoremA*}{Theorem A}
\newtheorem*{theoremB*}{Theorem B}
\newtheorem*{theoremC*}{Theorem C}
\newtheorem*{theoremD*}{Theorem D}
\newtheorem{definition}{Definition}[section]
\newtheorem{theorem}[definition]{Theorem}
\newtheorem{claim}[definition]{Claim}
\newtheorem{example}[definition]{Example}
\newtheorem{lemma}[definition]{Lemma}
\newtheorem{corollary}[definition]{Corollary}
\newtheorem{proposition}[definition]{Proposition}
\newtheorem{notation}[definition]{Notation}

\setcounter{tocdepth}{1}

\maketitle

\begin{abstract}
In this note, we study Shimura varieties for the groups $\GU(V)$, where $V$ is a Hermitian space relative to a CM extension $E/E^+$. We give a description of the supersingular locus of the fiber at a prime $\nu$ over $p$ of such a Shimura variety, under the assumptions that $\dim_E(V) \leq 4$ and that the prime $p$ splits completely in $E^+$ and is unramified in $E$.
\end{abstract}

\section{Introduction}\label{intro}

This note  gives a description of the supersingular loci of Shimura varieties for unitary groups relative to a CM extension  $E/E^+$, under  two conditions: First, we consider abelian varieties of a  fixed dimension $m[E^+:\Q]$, and we require that $m \leq 4$. Second,  this supersingular locus is defined by first taking the special fiber at a prime $\nu$ of the reflex field $F$, and we assume that $\nu$ is over a prime $p$ that \emph{splits completely} in $E^+$ and is \emph{unramified} in $E$.  

This description given in this note of such a supersingular locus follows  from combining the local results of Howard-Pappas ~\cite{GU22}, Vollaard ~\cite{V}, Vollaard-Wedhorn ~\cite{VW}, and the author ~\cite{GL4}, with other well-known results. To an expert in the field, it is likely apparent that the geometry of the supersingular locus at a prime $\nu$ under the above assumptions is determined by the local results in ~\cite{GU22}, ~\cite{V}, ~\cite{VW},  and ~\cite{GL4}.  However, no such description is currently available in the literature. As the problem described above is a preliminary case of the question of describing the supersingular loci of Shimura varieties for unitary groups relative to an arbitrary CM extension, and as this more general question may have interesting applications to related fields, it seems useful to record the geometry of the supersingular locus in this preliminary case.

 To give the precise statement, we must first introduce some notation: we have fixed a CM  field $E$ with totally real subfield $E^+$. Let $n = [E^+:\Q]$. We'll denote the $n$ embeddings of $E^+$ into $\overline{\Q}$ by $\{ \varphi_{i} \}_{i=1}^n$ and the $2n$ embeddings of $E$ into $\overline{\Q}$ by $\{ \varphi_{i,a} , \varphi_{i,b} \}_{i=1}^n$, where both $\varphi_{i,a}$ and $\varphi_{i,b}$ restrict to $\varphi_i$ on $E^+$. The nontrivial automorphism of $E$ over $E^+$ will be denoted as $e \mapsto \overline{e}$.

Let $p$ be an odd prime totally split in $E^+$ and unramified in $E$, let $m \leq 4$ be an integer, and let $\mathcal{O} \subset E$ be the integral closure of $\Z_{(p)}$ in $E$.  Fix a free $\mathcal{O}$-module $V$ of rank $m$ with a perfect $\mathcal{O}$-valued hermitian form $\langle \cdot , \cdot \rangle$. We'll denote the signature of $V \otimes_{\mathcal{O}} E$ at the place defined by $\varphi_i$ as $(a_i, b_i)$ (we will assume without loss of generality that $a_i \leq b_i$ for all $i$). Note that $a_i + b_i = m$ for all $i$. Let $G = \mathrm{GU}(V)$.  There is one technical note to make:  we must assume that not all of the signatures are $(0,m)$, in order for the moduli problem described below to be representable by a Shimura variety. 

 Let $K_p = G(\Z_p)$, choose a compact open subgroup $K^{p} \subset G(\mathbb{A}_p^f)$, and let $K = K_pK^p$. When $mn$ is even, the group $\mathrm{GU}(V \otimes_{\mathcal{O} } E)$ satisfies the Hasse principle, and so for sufficiently small $K^{p}$ there is a unitary Shimura variety $\mathrm{Sh}^K_{V}$ that parametrizes the prime-to-$p$ isogeny classes of quadruples $(A, \iota, \lambda, \eta^pK^p)$, where

\begin{itemize}
\item{ $A$ is an abelian variety over $\C$ of dimension $mn$.}
\item{ $\iota: \mathcal{O} \rightarrow \mathrm{End}(A) \otimes_{\Z} \Z_{(p)}$ is an action of $\mathcal{O}$ on $A$, satisfying the \emph{global signature $((a_i,b_i))_{i=1}^n$ condition}: for all $e \in \mathcal{O}$,
$$\det(T - \iota(e); \mathrm{Lie}(A) ) = \prod_{i=1}^n(T - \varphi_{i,a}(e))^{a_i}(T - \varphi_{i,b}(e))^{b_i}.$$}
\item{$\lambda \in \Hom(A, A^\vee) \otimes_{\Z} \Z_{(p)}$ is a prime-to-$p$ quasi-polarization of $A$, which satisfies the \emph{Rosati involution condition}: for all $e \in \mathcal{O}$,
$$\lambda \circ \iota(\overline{e}) = \iota(e)^\vee \circ \lambda,$$
where $\iota(e)^\vee$ is the dual isogeny to $\iota(e)$.}
\item{The level structure $\eta^pK^p$ is the $K^p$ orbit of an $\mathcal{O} \otimes \A_f^p$-linear isomorphism:
$$\eta^p: \widehat{T(A)}^p \otimes \A_f^p \rightarrow V \otimes \A_f^p$$
that respects the hermitian forms on either side, up to scaling by $(\A_f^p)^\times$.}
\end{itemize}

(If $mn$ is odd, the complex manifold parametrizing the above moduli problem is a disjoint union of finitely many copies of the Shimura variety $\mathrm{Sh}^K_{V}$.) Note that the terminology ``global signature condition" used above is not standard. The purpose of this phrase is to differentiate the signature condition in the description of the Shimura variety above from the signature conditions that will appear later in the descriptions of the relevant Rapoport-Zink spaces, which will be referred to as ``local signature conditions."

Let $F$ be the reflex field of the Shimura variety $\mathrm{Sh}^K_{V}$. For sufficiently small $K^p$, there is an integral canonical model $\mathcal{M}_{K}$ over $\mathcal{O}_{F} \otimes_{\Z} \Z_{(p)}$,  with the property that the complex manifold $\mathcal{M}_K(\C)$ parametrizes the moduli problem above.  (For details, see Kottwitz ~\cite{Ko}).

Let $\nu$ be a prime of $F$ over $p$, and fix an algebraic closure $k$ of $\mathcal{O}_{F}/ \nu \mathcal{O}_{F}$. We will denote by $\mathcal{M}_{K, \nu}^{ss}$ the reduced locus of $\mathcal{M}_K \times_{    \mathcal{O}_{F} \otimes_{\Z} \Z_{(p)}  } k $ parametrizing quadruples where the abelian variety $A$ is supersingular. This is called the supersingular locus. The choice of the prime $\nu$ determines a local signature condition, which in turn determines some parameters. The main result of this note is a description of $\mathcal{M}_{K, \nu}^{ss}$ in terms of these parameters.

\begin{theorem*}
Assume that $p$ splits completely in $E^+$ and is unramified in $E$. Fix a prime $\nu$ of the reflex field $F$ over $p$ and a sufficiently small compact open subgroup $K^p \subset G(\mathbb{A}_f^p)$. Let $(( a_i, b_i))_{i=1}^n$ be the tuple  defining the global signature condition for $\mathcal{M}_{K}$, and let  $( (a_{\nu, i}, b_{\nu,i}) )_{i=1}^n$ be the tuple of local signatures determined by $\nu$ as in Definition \ref{localsign}.  We will describe the geometry of the supersingular locus $\mathcal{M}_{K,\nu}^{ss}$ in the four cases $m=1$, $2$, $3$, and $4$. \\

\textbf{\underline{$\boldsymbol{m=1}:$} }\\
 If all $\mathfrak{p}_i$ are inert in $E$, the reduced $k$-scheme  $\mathcal{M}_{K,\nu}^{ss}$ is zero-dimensional.  If any of the primes $\mathfrak{p}_i$  are split in $E$, then  $\mathcal{M}_{K,\nu}^{ss}$  is empty. \\

\textbf{\underline{$\boldsymbol{m=2}:$} }\\
Define the following constant based on the prime $\nu$:
$$c_\nu  \ =  \ \text{ the number of } \  (a_{\nu,i}, b_{\nu,i}) = (0,2)  \ \text{ such that } \ \mathfrak{p_i} \ \text{ is split in } \ E.$$
If $c_\nu= 0$, the reduced $k$-scheme $\mathcal{M}_{K,\nu}^{ss}$ is zero-dimensional.  Otherwise, $\mathcal{M}_{K,\nu}^{ss}$ is empty. \\

\textbf{\underline{$\boldsymbol{m=3}:$} }\\
If any of the primes  $\mathfrak{p}_i$ are split in $E$, the supersingular locus  $\mathcal{M}_{K,\nu}^{ss}$ is empty. Otherwise, define the following constant which depends only on the global signature condition:
$$c \ =  \ \text{ the number of } \  (a_{i}, b_{i}) = (1,2) $$
Then the reduced $k$-scheme $\mathcal{M}_{K,\nu}^{ss}$  has geometry as follows:
\begin{itemize}
\item{The supersingular locus $\mathcal{M}_{K,\nu}^{ss}$ is equi-dimensional of dimension $c$.}
\item{Each irreducible component of  $\mathcal{M}_{K,\nu}^{ss}$ is isomorphic over $k$ to $C^{c}$,
 a product (over $k$) of $c$ copies of the Fermat curve $C$.}
 \item{Any two irreducible components either intersect trivially or have intersection isomorphic over $k$ to $C^{t}$
 for some integer $0 \leq t \leq c$. \\}
\end{itemize}

\textbf{\underline{$\boldsymbol{m=4}:$} }\\
Define the following constants based on the prime $\nu$:
$$c_\nu \ =  \ \text{ the number of } \  (a_{\nu,i}, b_{\nu,i}) = (0,4)  \ \text{ such that } \ \mathfrak{p_i} \ \text{ is inert in } \ E$$
$$d_\nu \ =  \ \text{ the number of } \   (a_{\nu,i}, b_{\nu,i})  = (1,3)  \ \text{ such that } \ \mathfrak{p_i} \ \text{ is inert in } \ E$$
$$e_\nu \ =  \ \text{ the number of } \   (a_{\nu,i}, b_{\nu,i})  = (2,2)  \ \text{ such that } \ \mathfrak{p_i} \ \text{ is inert in } \ E$$
$$f_\nu \ =  \ \text{ the number of } \   (a_{\nu,i}, b_{\nu,i})  = (2,2)  \ \text{ such that } \ \mathfrak{p_i} \ \text{ is split in } \ E$$
$$g_\nu \ = \ \text{ the number of } \   (a_{\nu,i}, b_{\nu,i}) \neq (2,2)  \ \text{ such that } \ \mathfrak{p_i} \ \text{ is split in } \ E.$$
Note that any of these integers could be zero, and that $c_\nu + d_\nu + e_\nu + f_\nu + g_\nu = n$. Either $g_\nu > 0$, in which case $\mathcal{M}_{K,\nu}^{ss}$ is empty, or $g_\nu = 0$, in which case $\mathcal{M}_{K,\nu}^{ss}$ has geometry as follows:

\begin{itemize}
\item{The supersingular locus   $\mathcal{M}_{K, \nu}^{ss}$ is equi-dimensional  of dimension $d_\nu+2e_\nu+f_\nu$.\\}

\item{Each irreducible component of $\mathcal{M}_{K, \nu}^{ss}$ is isomorphic over $k$ to:
$$ C^{d_\nu} \times_k  S^{e_\nu} \times_k (\mathbb{P}^1)^{f_\nu}$$ 
 a product (over $k$) of $d_\nu$ copies of the Fermat curve,  $e_\nu$ copies of the Fermat surface, and  $f_\nu$ copies of the projective line. \\}
 
\item{ Any two irreducible components either intersect trivially or have intersection isomorphic over $k$ to:
 $$ C^r \times_k S^{s_1} \times_k (\mathbb{P}^1)^{s_2 +t}$$
 for some nonnegative integers $r, s_1, s_2, t$ satisfying the following three conditions:
\begin{enumerate}
\item{ $$0 \leq r \leq d_\nu$$ }
\item{ $$0 \leq s_1 + s_2 \leq e_\nu$$ }
\item{ $$0 \leq t \leq f_\nu.$$ }
\end{enumerate} }
\end{itemize}
\end{theorem*}

This result follows from combining the Rapoport-Zink Uniformization Theorem and the local results of  ~\cite{GU22},  ~\cite{V},  ~\cite{VW},  and  ~\cite{GL4}.  The intuition behind this is that the unitary Shimura varieties in this note parametrize abelian varieties with an action of $E/E^+$ and some extra structure. By the Rapoport-Zink Uniformization Theorem, we may reduce the study of the supersingular loci to the study of the relevant Rapoport-Zink spaces, which are moduli spaces of $p$-divisible groups with an action of $E \otimes_{\Q} \Q_p / E^+ \otimes_{\Q} \Q_p$. Depending on the factorization of $p$ in $E$, this extension decomposes as a product of some number of copies of the quadratic unramified extension $L / \Q_p$ and some number of copies of the ``split" extension $\Q_p \times \Q_p / \Q_p$.

The papers ~\cite{GU22}, ~\cite{V}, and ~\cite{VW}  give a description of the geometry of some Rapoport-Zink spaces parametrizing $p$-divisible groups with actions of $L / \Q_p$ of signature $(2,2)$, $(1,2)$, and $(1,3)$, respectively, while ~\cite{GL4} gives a description of a Rapoport-Zink space parametrizing $p$-divisible groups with an action of $\Q_p \times \Q_p / \Q_p$ of signature $(2,2)$.  Depending on the prime $\nu$, these local results may be combined to describe the geometry of the Rapoport-Zink space relevant to the study of the supersingular locus $\mathcal{M}_{K,\nu}^{ss}$. Then, the Rapoport-Zink Uniformization Theorem may be used to produce the above description of $\mathcal{M}_{K,\nu}^{ss}$.

\subsection*{Acknowledgements} I would like to thank  Ellen Eischen, Elena Mantovan, and Rachel Pries for interesting conversations regarding the topics discussed in this paper. I would also like to thank  Ben Howard and Keerthi Madapusi Pera for their comments.

\section{Unitary Rapoport-Zink Spaces}\label{RZdef}

In this section, we will define the Rapoport-Zink spaces relevant to the study of the supersingular loci $\mathcal{M}_{K,\nu}^{ss}$. We will begin with some notation: we've assumed the prime $p$ splits completely in $E^+$, so let $p \mathcal{O}_{E^+} = \mathfrak{p}_1 \cdots \mathfrak{p_n}$. Because the prime $p$ is split completely in $E^+$, there are $n$ distinct embeddings $\{ \psi_i \}_{i=1}^n$ of $E^+$ into $\overline{\Q}_p$, which all have image in $\Q_p$. Viewed in another way, these $n$ embeddings give $n$ distinct places of $E^+$, corresponding to the $n$ distinct prime ideals in the factorization of $p\mathcal{O}_{E^+}$. Label the primes in the factorization of $p\mathcal{O}_{E^+}$ so that $\mathfrak{p}_i$ corresponds to $\psi_i$.

We've also assumed that $p$ does not ramify in $E$. Relabel the primes $\mathfrak{p}_i$ and the corresponding embeddings $\psi_i$ so that $\mathfrak{p}_i$ splits in $E$ for all $i \leq h$ and $\mathfrak{p}_i$ is inert in $E$ for all  $i>h$ (note that $h$ might be 0 or $n$). If $\mathfrak{p}_i$ is split in $E$, label the two primes above $\mathfrak{p}_i$ as $\mathfrak{p}_{i,a}$ and $\mathfrak{p}_{i,b}$. If the prime  $\mathfrak{p}_i$ is inert in $E$, we'll also use the notation  $\mathfrak{p}_i$ to denote a prime ideal of $E$. Label the $2n$ distinct embeddings of $E$ into $\overline{\Q}_p$ as $\{ \psi_{i,a}, \psi_{i,b} \}_{i=1}^n$, where  the embeddings $\psi_{i,a}$ and $\psi_{i,b}$ correspond respectively to the prime ideals $\mathfrak{p}_{i,a}$ and $\mathfrak{p}_{i,b}$ whenever $i \leq h$, and   both $\psi_{i,a}$ and $\psi_{i,b}$ correspond to the prime ideal $\mathfrak{p}_i$ whenever $i > h$. Note that the embeddings $\psi_{i,a}$ and $\psi_{i,b}$  have image in $\Q_p$ whenever $i \leq h$ and form a conjugate pair of embeddings into the quadratic unramified extension $L$ of $\Q_p$ whenever $i > h$.

We can understand $\mathcal{O}_{E^+} \otimes_{\Z} \Z_p$ in terms of these embeddings. We have an isomorphism:
$$\mathcal{O}_{E^+} \otimes_{\Z} \Z_p \cong \Z_p \times  \cdots \times \Z_p$$
$$e \otimes z \mapsto (\psi_i(e) z, \dots, \psi_n(e) z).$$

Similarly, there is an isomorphism:
$$\mathcal{O}_E \otimes_{\Z} \Z_p \cong  \underbrace{ (\Z_p \times \Z_p) \times \cdots \times (\Z_p \times \Z_p) }_{h} \times \underbrace{\mathcal{O}_L \times \cdots \times \mathcal{O}_L}_{n-h}.$$
$$e \otimes z \mapsto ( (\psi_{i,a}(e)z, \psi_{i,b}(e)z)_{i=1}^h,  (  \psi_{i,a}(e)z)_{i=h+1}^n ).$$

In this way, we can view $\mathcal{O}_E \otimes_{\Z} \Z_p$ over $\mathcal{O}_{E^+} \otimes_{\Z} \Z_p$ as a product of $h$ copies of the split extension $\Z_p \times \Z_p$ over $\Z_p$ and $n-h$ copies of the quadratic unramified extension $\mathcal{O}_L$ over $\Z_p$. Note that the automorphism $e \mapsto \overline{e}$ of $E$ over $E^+$ permutes the two $\Z_p$-factors of each copy of $\Z_p \times \Z_p$ over $\Z_p$ and  induces the nontrivial automorphism of each copy of  $\mathcal{O}_L$ over $\Z_p$.

Let $k$ be an algebraically closed field of characteristic $p$, and let $W = W(k)$ be the ring of Witt vectors of $k$, with fraction field $W_{\Q}$.  To define the Rapoport-Zink spaces needed to uniformize the supersingular locus $M_{K,\nu}^{ss}$, we will begin by fixing a basepoint object. Let $\boldsymbol{G}$ be a supersingular $p$-divisible group over $k$ of dimension $mn$ and height $2mn$, with an action $\boldsymbol{\iota}: \mathcal{O}_E \otimes_{\Z} \Z_p \rightarrow \End(\boldsymbol{G})$ and a principal polarization  $\boldsymbol{\lambda}: \boldsymbol{G} \rightarrow \boldsymbol{G^\vee}$. Assume that the action $\biota$ satisfies the \emph{local signature $( (a_i, b_i) )_{i=1}^n$ condition}:
$$\mathrm{det}(T - \biota(e); \mathrm{Lie}(\bG) ) = \prod_{i=1}^n(T - \psi_{i,a}(e))^{a_i}(T - \psi_{i,b}(e))^{b_i},$$
where the right hand side is reduced modulo $p$ to yield a polynomial with coefficients in $k$, and each $a_i + b_i = m$.  Also assume that the action $\biota$ and polarization $\blambda$ together satisfy the \emph{Rosati involution condition}:
$$\blambda \circ \biota(\overline{e} ) = \biota(e)^\vee \circ \blambda$$
for any $e \in \mathcal{O}_E \otimes_{\Z} \Z_{p}$, where $\biota(e)^\vee$ is the dual isogeny to $\biota(e)$.

There are two important notes to make: first, we'll see later that a triple $(\bG, \biota, \blambda)$ satisfying the above conditions only exists when certain conditions on the signature are met. Second, this local signature condition does resemble the global signature condition described in the introduction, but the local signature condition involves the embeddings $\{ \psi_{i,a}, \psi_{i,b} \}_{i=1}^n$ of $E$ into $\overline{\Q}_p$, not the embeddings $\{ \varphi_{i,a}, \varphi_{i,b} \}_{i=1}^n$ of $E$ into $\overline{\Q}$. In particular, we are not assuming that the ordered tuple $( (a_i, b_i) )_{i=1}^n$ used here to define the local signature condition is equal to the tuple used in the introduction to define the global signature condition.

Let $\mathrm{Nilp}_W$ be the category of $W$-schemes on which $p$ is locally nilpotent. For a scheme $S$ in $\mathrm{Nilp}_W$, we will use the notation $S_0 = S \otimes_{W} k$.

\begin{definition}
Let $\widetilde{\mathcal{N}}_{( (a_i, b_i) )_{i=1}^n}$ be the functor on $\mathrm{Nilp}_W$ that sends a connected $S$ in $\mathrm{Nilp}_W$ to the collection of isomorphism classes of quadruples $(G, \iota, \lambda, \rho)$, where:
\begin{itemize}
\item{$G$ is a supersingular $p$ divisible group over $S$ of dimension $mn$ and height $2mn$}
\item{ $\iota:  \mathcal{O}_E \otimes_{\Z} \Z_p \rightarrow \End(G)$ is an action of $\mathcal{O}_E \otimes_{\Z} \Z_p$}
\item{$\lambda: G \rightarrow G^\vee$ is a principal polarization}
\item{$\rho: G_{S_0} \rightarrow \bG_{S_0}$ is an $\mathcal{O}_E \otimes_{\Z} \Z_p$-linear quasi-isogeny}
\end{itemize}
such that the action $\iota$ satisfies the \emph{local signature $( (a_i, b_i) )_{i=1}^n$ condition}:
$$\mathrm{det}(T - \biota(e); \mathrm{Lie}(G) )= \prod_{i=1}^n(T - \psi_{i,a}(e))^{a_i}(T - \psi_{i,b}(e))^{b_i},$$
 where all $a_i + b_i = m$, together $\iota$ and $\lambda$ satisfy the \emph{Rosati involution condition},
 $$\lambda \circ \iota(\overline{e} ) = \iota(e)^\vee \circ \lambda,$$
 and such that $\rho^*(\blambda) = c(\rho) \lambda$ for some $c(\rho) \in \Q_p^\times$.  Two quadruples $(G_1, \iota_1, \lambda_1, \rho_1)$ and $(G_2, \iota_2, \lambda_2, \rho_2)$ are considered to be isomorphic if there is an isomorphism $\Phi: G_1 \rightarrow G_2$ of $p$-divisible groups over $S$, that is  $\mathcal{O}_E \otimes_{\Z} \Z_p$-linear, that carries $\rho_2$ to $\rho_1$, and such that $\lambda_2$ pulls back to a $\Z_p^\times$-multiple of $\lambda_1$.
 \end{definition}

By Rapoport-Zink ~\cite{RZ}, this functor is represented by a formal scheme $\widetilde{\mathcal{N} }_{( (a_i, b_i) )_{i=1}^n}$  over $\mathrm{Spf}(W)$. Let $\mathcal{N}_{( (a_i, b_i) )_{i=1}^n}$ be the underlying reduced scheme, which is a scheme over $k$.

The  local objective of this note is to relate this Rapoport-Zink space to those studied in \cite{GU22}, \cite{V}, \cite{VW},  , and \cite{GL4}. So we will now introduce the notation necessary to refer to the Rapoport-Zink spaces studied in those papers:

Let $p,q \geq 0$ be integers such that $p+q = m$ (without loss of generality, we'll always assume $p \leq q$). We will first define a unitary Rapoport-Zink space parametrizing $p$-divisible groups with an action by $\mathcal{O}_L$ of local signature $(p,q)$. Let $\psi_{L,1}$ and $\psi_{L,2}$ be the conjugate pair of $\Z_p$ morphisms from $L$ to $W$. Let $\boldsymbol{G}^{\mathrm{inert}}_{(p,q)}$ be a supersingular $p$-divisible group over $k$ of dimension $m$ and height $2m$, with an action $\boldsymbol{\iota}^{\mathrm{inert}}_{(p,q)}: \mathcal{O}_L  \rightarrow \End( \boldsymbol{G}^{\mathrm{inert}}_{(p,q)} )$ and a principal polarization  $\boldsymbol{\lambda}_{(p,q)}^{\mathrm{inert}}: \boldsymbol{G}^{\mathrm{inert}}_{(p,q)}  \rightarrow (\boldsymbol{G}^{\mathrm{inert}}_{(p,q)})^\vee$. Assume for now that the action $\biota^{\mathrm{inert}}_{(p,q)}$ satisfies the  \emph{local signature $(p,q)$ condition}:
$$\mathrm{det}(T - \biota^{\mathrm{inert}}_{(p,q)}(l); \mathrm{Lie}(\bG^{\mathrm{inert}}_{(p,q)}) ) = (T - \psi_{L,1}(l))^{p}(T - \psi_{L,2}(l))^{q},$$
where the right hand side is reduced modulo $p$ to yield a polynomial with coefficients in $k$,  and that the action $\biota^{\mathrm{inert}}_{(p,q)}$ and polarization $\boldsymbol{\lambda}_{(p,q)}^{\mathrm{inert}}$ together satisfy the \emph{Rosati involution condition}:
$$\boldsymbol{\lambda}_{(p,q)}^{\mathrm{inert}} \circ \biota^{\mathrm{inert}}_{(p,q)}(\overline{l} ) = \biota^{\mathrm{inert}}_{(p,q)}(l)^\vee \circ \boldsymbol{\lambda}_{(p,q)}^{\mathrm{inert}}$$
for any $l \in \mathcal{O}_L$, where $\biota^{\mathrm{inert}}_{(p,q)}(l)^\vee$ is the dual isogeny to $\biota^{\mathrm{inert}}_{(p,q)}(l)$. 

\begin{definition}
Let $\widetilde{\mathcal{N}}^{\mathrm{inert}}_{(p,q)}$ be the functor on $\mathrm{Nilp}_W$ that sends a connected $S$ in $\mathrm{Nilp}_W$ to the collection of isomorphism classes of quadruples $(G, \iota, \lambda, \rho)$, where 
\begin{itemize}
\item{$G$ is a supersingular $p$ divisible group of dimension $m$ and height $2m$}
\item{$\iota: \mathcal{O}_L \rightarrow \End(G)$ is an action of $\mathcal{O}_L$ meeting the \emph{local signature $(p,q)$ condition}}
\item{$\lambda: G \rightarrow G^\vee$ is a principal polarization satisfying the \emph{Rosati involution condition}}
\item{$\rho: G_{S_0} \rightarrow  (\bG^{\mathrm{inert}}_{(p,q)})_{S_0}$ is an $\mathcal{O}_L$-linear quasi-isogeny}
\end{itemize}
such that $\rho^*(  \boldsymbol{\lambda}_{(p,q)}^{\mathrm{inert}} ) = c(\rho) \lambda$ for some $c(\rho) \in \Q_p^\times$.  Two quadruples $(G_1, \iota_1, \lambda_1, \rho_1)$ and $(G_2, \iota_2, \lambda_2, \rho_2)$ are considered to be isomorphic if there is an isomorphism $\Phi: G_1 \rightarrow G_2$ of $p$-divisible groups over $S$ that is  $\mathcal{O}_L$-linear, that carries $\rho_2$ to $\rho_1$, and such that $\lambda_2$ pulls back to a $\Z_p^\times$-multiple of $\lambda_1$.
\end{definition}
By Rapoport-Zink ~\cite{RZ}, this is represented by a formal scheme $\widetilde{\mathcal{N}}^{\mathrm{inert}}_{(p,q)}$  over $\mathrm{Spf}(W)$. See Kudla-Rapoport ~\cite{local} for further details. We'll denote the underlying reduced $k$-scheme by $\mathcal{N}^{\mathrm{inert}}_{(p,q)}$.

For any integers $p,q \geq 0$ such that $p+q = m$, we can similarly define a unitary Rapoport-Zink space with action by $\Z_p \times \Z_p$ of signature $(p,q)$. Let $\psi_{\Z_p,1}$ and $\psi_{\Z_p,2}$ be the two $\Z_p$-morphisms from $\Z_p \times \Z_p$ to $\Z_p$ by projecting from the first or second factors, respectively. We will use the same notation for these two morphisms taken with values in $W$. Define a basepoint $( \boldsymbol{G}^{\mathrm{split}}_{(p,q)}, \boldsymbol{\iota}^{\mathrm{split}}_{(p,q)}, \boldsymbol{\lambda}_{(p,q)}^{\mathrm{split}})$ similar to the above, where now $\boldsymbol{\iota}^{\mathrm{split}}_{(p,q)}$ is an action of $\Z_p \times \Z_p$ meeting the \emph{local signature $(p,q)$ condition} formulated with $\psi_{\Z_p,1}$ and $\psi_{\Z_p,2}$, and together $\boldsymbol{\iota}^{\mathrm{split}}_{(p,q)}$ and $ \boldsymbol{\lambda}^{\mathrm{split}}_{(p,q)}$ satisfy the \emph{Rosati involution condition}.

\begin{definition}
Let $\widetilde{\mathcal{N}}^{\mathrm{split}}_{(p,q)}$ be the functor on $\mathrm{Nilp}_W$ that sends a connected $S$ in $\mathrm{Nilp}_W$ to the collection of isomorphism classes of quadruples $(G, \iota, \lambda, \rho)$, where 
\begin{itemize}
\item{$G$ is a supersingular $p$ divisible group of dimension $m$ and height $2m$}
\item{$\iota: \Z_p \times \Z_p \rightarrow \End(G)$ is an action of $\Z_p \times \Z_p$ meeting the \emph{local signature $(p,q)$ condition}}
\item{$\lambda: G \rightarrow G^\vee$ is a principal polarization satisfying the \emph{Rosati involution condition}}
\item{$\rho: G_{S_0} \rightarrow  (\bG^{\mathrm{split}}_{(p,q)})_{S_0}$ is an $\Z_p \times \Z_p$-linear quasi-isogeny}
\end{itemize}
such that $\rho^*(  \boldsymbol{\lambda}_{(p,q)}^{\mathrm{split}} ) = c(\rho) \lambda$ for some $c(\rho) \in \Q_p^\times$. Two quadruples $(G_1, \iota_1, \lambda_1, \rho_1)$ and $(G_2, \iota_2, \lambda_2, \rho_2)$ are considered to be isomorphic if there is an isomorphism $\Phi: G_1 \rightarrow G_2$ of $p$-divisible groups over $S$ that is  $\Z_p \times \Z_p$-linear, that carries $\rho_2$ to $\rho_1$, and such that $\lambda_2$ pulls back to a $\Z_p^\times$-multiple of $\lambda_1$.
\end{definition}
By Rapoport-Zink ~\cite{RZ}, this is represented by a formal scheme $\widetilde{\mathcal{N}}^{\mathrm{split}}_{(p,q)}$  over $\mathrm{Spf}(W)$. We'll denote the underlying reduced $k$-scheme by $\mathcal{N}^{\mathrm{split}}_{(p,q)}$ .

\section{Decomposition}
In this section, we will state and prove the main local result, that $\mathcal{N}_{( (a_i, b_i) )_{i=1}^n}$ may be described in terms of the $\mathcal{N}^{\mathrm{inert}}_{(p,q)}$ and $\mathcal{N}^{\mathrm{split}}_{(p,q)}$.

Let $\widetilde{\mathcal{N}}^{j}_{( (a_i, b_i) )_{i=1}^n}$ be the open and closed formal subscheme of $\widetilde{ \mathcal{N} }_{( (a_i, b_i) )_{i=1}^n}$ where $\ord_p( c(\rho) ) = j$, and define $\widetilde{\mathcal{N}}^{\mathrm{split},j}_{(p,q)}$ and $\widetilde{\mathcal{N}}^{\mathrm{inert},j}_{(p,q)}$ similarly. We have decompositions:
$$\widetilde{\mathcal{N}}_{( (a_i, b_i) )_{i=1}^n} = \bigsqcup_{j \in \Z}  \widetilde{\mathcal{N}}^{j}_{( (a_i, b_i) )_{i=1}^n} \quad \quad \quad \widetilde{\mathcal{N}}^{\mathrm{inert}}_{(p,q)} =  \bigsqcup_{j \in \Z}  \widetilde{\mathcal{N}}^{\mathrm{inert},j}_{(p,q)} \quad  \quad \quad \widetilde{\mathcal{N}}^{\mathrm{split}}_{(p,q)} =  \bigsqcup_{j \in \Z}  \widetilde{\mathcal{N}}^{\mathrm{split},j}_{(p,q)}.$$

\begin{theorem}\label{Decomp}
Recall that we've assumed that the primes $\mathfrak{p}_i$ are split in $E$ for $i \leq h$ and are inert in $E$ for $i > h$. There is an isomorphism
$$\widetilde{\mathcal{N} }^{j}_{( (a_i, b_i) )_{i=1}^n} \cong \widetilde{\mathcal{N}}^{\mathrm{split},j}_{(a_1,b_1)} \times \cdots \times \widetilde{\mathcal{N}}^{\mathrm{split},j}_{(a_h,b_h)} \times  \widetilde{\mathcal{N}}^{\mathrm{inert},j}_{(a_{h+1},b_{h+1})} \times \cdots \times \widetilde{\mathcal{N}}^{\mathrm{inert},j}_{(a_n,b_n)}$$ 
In the case that $h = 0$ or $h = n$, the above statement should be taken to mean
$$\widetilde{\mathcal{N} }_{( (a_i, b_i) )_{i=1}^n}^{j} \cong \widetilde{\mathcal{N}}^{\mathrm{inert},j}_{(a_{1},b_{1})} \times \cdots \times \widetilde{\mathcal{N}}^{\mathrm{inert},j}_{(a_n,b_n)}   \quad \text{ or } \quad   \widetilde{\mathcal{N} }_{( (a_i, b_i) )_{i=1}^n}^{j} \cong \widetilde{\mathcal{N}}^{\mathrm{split},j}_{(a_1,b_1)} \times \cdots \times \widetilde{\mathcal{N}}^{\mathrm{split},j}_{(a_n,b_n)},$$
respectively. The products above are all over $\mathrm{Spf}(W)$.
\end{theorem}

\begin{proof}
 We will construct a natural morphism and its inverse between the $S$-points of these functors for any connected $S$ in $\mathrm{Nilp}_W$. We will use notation implying that $1 \leq h \leq n-1$, but the analogous statements are clear for $h=0$ and $h=n$. 

Recall that $\bG$ is a basepoint $p$-divisible group for $\widetilde{\mathcal{N}}_{( (a_i, b_i) )_{i=1}^n}$, which is equipped with an action  $\boldsymbol{\iota}: \mathcal{O}_E \otimes_{\Z} \Z_p \rightarrow \End(\boldsymbol{G})$ satisfying the local signature $( (a_i, b_i) )_{i=1}^n$ condition and a principal polarization  $\boldsymbol{\lambda}: \boldsymbol{G} \rightarrow \boldsymbol{G^\vee}$.   Define a new choice of basepoint $p$-divisible group $\boldsymbol{G}$ for the functor $\widetilde{\mathcal{N}}_{( (a_i, b_i) )_{i=1}^n}$ as:

$$\boldsymbol{G} = \boldsymbol{G}^{\mathrm{split}}_{(a_1,b_1)} \times_k \cdots \times_k \boldsymbol{G}^{\mathrm{split}}_{(a_h,b_h)} \times_k  \boldsymbol{G}^{\mathrm{inert}}_{(a_{h+1},b_{h+1})} \times_k \cdots \times_k \boldsymbol{G}^{\mathrm{inert}}_{(a_n,b_n)}$$

Note that, as each p-divisible group $\boldsymbol{G}^{\mathrm{split}}_{(a_i,b_i)}$ or $\boldsymbol{G}^{\mathrm{inert}}_{(a_i,b_i)}$ is supersingular of height $2m$ and dimension $m$, this  $p$-divisible group $\boldsymbol{G}$ is supersingular of height $2mn$ and dimension $mn$. 

Each of the $\boldsymbol{G}^{\mathrm{split}}_{(a_i,b_i)}$ (resp. $\boldsymbol{G}^{\mathrm{inert}}_{(a_i,b_i)}$) have their own action $\boldsymbol{\iota}^{\mathrm{split}}_{(a_i,b_i)}$ (resp. $\boldsymbol{\iota}^{\mathrm{inert}}_{(a_i,b_i)}$). Define an action $\biota$ on $
\boldsymbol{G}$ as:

$$\biota = (\boldsymbol{\iota}^{\mathrm{split}}_{(a_1,b_1)}, \cdots, \boldsymbol{\iota}^{\mathrm{split}}_{(a_h,b_h)}, \boldsymbol{\iota}^{\mathrm{inert}}_{(a_{h+1},b_{h+1})}, \cdots \boldsymbol{\iota}^{\mathrm{inert}}_{(a_n,b_n)} )$$
as the product of these actions on the respective factors of $\boldsymbol{G}$. This is a map 
$$  (\Z_p \times \Z_p)^h  \times \mathcal{O}_L^{n-h} \rightarrow \mathrm{End}(\boldsymbol{G}),$$
but through the isomorphism 
$$\mathcal{O}_E \otimes_{\Z} \Z_p \cong  (\Z_p \times \Z_p)^h  \times \mathcal{O}_L^{n-h}$$
$$e \otimes z \mapsto ( (\psi_{i,a}(e)z, \psi_{i,b}(e)z)_{i=1}^h,  (  \psi_{i,a}(e)z)_{i=h+1}^n )$$
it also defines an action of $\mathcal{O}_E \otimes_{\Z} \Z_p$ on $\boldsymbol{G}$.  As each of the actions $\boldsymbol{\iota}^{\mathrm{split}}_{(a_i,b_i)}$ and $\boldsymbol{\iota}^{\mathrm{inert}}_{(a_i,b_i)}$ satisfy their respective local signature $(a_i, b_i)$ condition, the combined action $\biota$ satisfies the local signature $( (a_i, b_i) )_{i=1}^n$ condition by construction.

Finally, we'll construct a principal polarization $\blambda$ also as a product:
$$\blambda = ( \boldsymbol{\lambda}_{(a_1,b_1)}^{\mathrm{split}} , \cdots, \boldsymbol{\lambda}_{(a_h, b_h)}^{\mathrm{split}} , \boldsymbol{\lambda}_{(a_{h+1},b_{h+1})}^{\mathrm{inert}} , \cdots, \boldsymbol{\lambda}_{(a_n,b_n)}^{\mathrm{inert}}).$$

Note that as each $ \boldsymbol{\lambda}_{(a_i,b_i)}^{\mathrm{split}}$ (resp. $\boldsymbol{\lambda}_{(a_i,b_i)}^{\mathrm{inert}}$) satisfies the Rosati involution condition for the action  $\boldsymbol{\iota}^{\mathrm{split}}_{(a_i,b_i)}$ (resp. $\boldsymbol{\iota}^{\mathrm{inert}}_{(a_i,b_i)}$), the product polarization $\blambda$ satisfies the Rosati involution condition for the product action $\biota$.

As there is a single isogeny class of such objects (see ~\cite{VW} Lemma 5.1), we may without loss of generality use this choice of $(\boldsymbol{G}, \biota, \blambda)$ as the basepoint in the definition of $\widetilde{\mathcal{N}}_{( (a_i, b_i) )_{i=1}^n}$. 

Now, consider a connected $S$ in $\mathrm{Nilp}_W$. We will construct:
$$\Phi: \widetilde{\mathcal{N} }_{( (a_i, b_i) )_{i=1}^n}^j(S) \rightarrow  \widetilde{\mathcal{N}}^{\mathrm{split},j}_{(a_1,b_1)}(S) \times \cdots \times \widetilde{\mathcal{N}}^{\mathrm{split},j}_{(a_h,b_h)}(S) \times  \widetilde{\mathcal{N}}^{\mathrm{inert},j}_{(a_{h+1},b_{h+1})}(S) \times \cdots \times \widetilde{\mathcal{N}}^{\mathrm{inert},j}_{(a_n,b_n)}(S).$$

Let $(G, \iota, \lambda, \rho) \in  \widetilde{\mathcal{N} }_{( (a_i, b_i) )_{i=1}^n}^j(S)$, and recall that we have an isomorphism:
 $$\mathcal{O}_{E^+} \otimes_{\Z} \Z_p \cong \Z_p \times  \cdots \times \Z_p$$
$$e \otimes z \mapsto (\psi_1(e) z, \dots, \psi_n(e) z),$$
so let $\epsilon_1, \cdots, \epsilon_n \in \mathcal{O}_{E^+} \otimes_{\Z} \Z_p$ be the preimages of the $n$ idempotent elements on the right-hand side. We may view $\mathcal{O}_{E^+} \otimes_{\Z} \Z_p \subset \mathcal{O}_E \otimes_{\Z} \Z_p$, so we can also view the $\epsilon_i$ as elements of $\mathcal{O}_E \otimes_{\Z} \Z_p$. The key observation is that:
$$G \cong \iota(\epsilon_1) G \times_S  \cdots \times_S \iota(\epsilon_n) G$$
where the action of $\mathcal{O}_{E^+} \otimes_{\Z} \Z_p$ acts through $\psi_i$ on $\iota(\epsilon_i)G$, and so the action $\iota$ of
$$\mathcal{O}_E \otimes_{\Z} \Z_p \cong   (\Z_p \times \Z_p)^h  \times \mathcal{O}_L^{n-h}$$
on $\iota(\epsilon_i)G$ either factors through $\Z_p \times \Z_p$ (if $i \leq h$) or through $\mathcal{O}_L$ (if $i > h$).

 Further, because $\rho$ is $\mathcal{O}_E \otimes_{\Z} \Z_p$-linear, it restricts to a quasi-isogeny:
$$\rho_i: (\iota(\epsilon_i)G)_{S_0} \rightarrow (\boldsymbol{G}^{\mathrm{split}}_{(a_i,b_i)})_{S_0}$$
for all $i \leq h$, and restricts to a quasi-isogeny:
$$\rho_i: (\iota(\epsilon_i)G)_{S_0} \rightarrow (\boldsymbol{G}^{\mathrm{inert}}_{(a_i,b_i)})_{S_0}$$
for all $i > h$. 

Finally, because $\lambda$ satisfies the Rosati involution condition, and because the automorphism $e \mapsto \overline{e}$ restricts to the identity on $\mathcal{O}_{E^+} \otimes_{\Z} \Z_p$, the polarization $\lambda$ restricts to principal polarizations:
$$\lambda_i: \iota(\epsilon_i)G  \rightarrow \iota(\epsilon_i)G^\vee$$
for all $i$. Since $\rho^*(\blambda) = c \lambda$, we have that $\rho_i^*( \boldsymbol{\lambda}_{(a_i,b_i)}^{\mathrm{split}})$ (resp. $\rho_i^*( \boldsymbol{\lambda}_{(a_i,b_i)}^{\mathrm{inert}}$)) is equal to $c \lambda_i$ for all $i$, for a fixed $c \in \Q_p^\times$ with $\ord_p(c) = j$. 

Because $(G, \iota, \lambda, \rho)$ satisfies the local signature $( (a_i, b_i ) )_{i=1}^n$ and Rosati involution conditions, we have that for all $i$ the quadruple $(\iota(\epsilon_i)G, \iota_i, \lambda_i, \rho_i)$ meets the local signature $(a_i, b_i)$ condition and the Rosati involution condition.  In particular, as $G$ is supersingular of height $2mn$ and dimension $mn$, the local signature condition implies that each $G_i$ is supersingular of height $2m$ and dimension $m$.

Finally, note that if $(G, \iota, \lambda, \rho)$ is replaced by another representative $(G', \iota', \lambda', \rho')$ in its isomorphism class, there is by definition an isomorphism $\Phi: G \rightarrow G'$ of $p$-divisible groups over $S$, that is  $\mathcal{O}_E \otimes_{\Z} \Z_p$-linear and respects the polarizations and quasi-isogenies. Precisely because this isomorphism is $\mathcal{O}_E \otimes_{\Z} \Z_p$-linear, replacing $(G, \iota, \lambda, \rho)$ with $(G', \iota', \lambda', \rho')$ in the above construction will produce quadruples isomorphic to the quadruples $(\iota(\epsilon_i)G, \iota_i, \lambda_i, \rho_i)$ for all $i$.  

Therefore, we have constructed a map:

$$\Phi: \widetilde{\mathcal{N} }_{( (a_i, b_i) )_{i=1}^n}^j(S) \rightarrow  \widetilde{\mathcal{N}}^{\mathrm{split},j}_{(a_1,b_1)}(S) \times \cdots \times \widetilde{\mathcal{N}}^{\mathrm{split},j}_{(a_h,b_h)}(S) \times  \widetilde{\mathcal{N}}^{\mathrm{inert},j}_{(a_{h+1},b_{h+1})}(S) \times \cdots \times \widetilde{\mathcal{N}}^{\mathrm{inert},j}_{(a_n,b_n)}(S) $$

$$(G, \iota, \lambda, \rho) \mapsto ( (\iota(\epsilon_i)G, \iota_i, \lambda_i, \rho_i) )_{i=1}^n.$$

Now we will construct the inverse $\Psi$ of $\Phi$ for a connected $S$.  The most important observation is that given an $n$-tuple $( (G_i, \iota_i, \lambda_i, \rho_i) )_{i=1}^n$ in 
$$ \widetilde{\mathcal{N}}^{\mathrm{split},j}_{(a_1,b_1)}(S) \times \cdots \times \widetilde{\mathcal{N}}^{\mathrm{split},j}_{(a_h,b_h)}(S) \times  \widetilde{\mathcal{N}}^{\mathrm{inert},j}_{(a_{h+1},b_{h+1})}(S) \times \cdots \times \widetilde{\mathcal{N}}^{\mathrm{inert},j}_{(a_n,b_n)}(S), $$
we have that $\rho_i^*(\boldsymbol{\lambda}_{(a_i,b_i)}^{\mathrm{split}}) = c_i \lambda_i$ (resp. $\rho_i^*(\boldsymbol{\lambda}_{(a_i,b_i)}^{\mathrm{inert}}) = c_i \lambda_i$) for all $i$, where each $c_i \in \Q_p^\times$ satisfies $\ord_p(c_i) = j$, but the $c_i$ might not all be the same element of $\Q_p^\times$. However, for all $i$, the polarization $\frac{c_i}{c_1} \lambda_i$ is again a principal polarization satisfying the necessary conditions. The identity map defines an isomorphism of quadruples:
$$(G_i, \iota_i, \lambda_i, \rho_i) \cong (G_i, \iota_i, \frac{c_i}{c_1} \lambda_i, \rho_i)$$
and $\rho_i^*(\boldsymbol{\lambda}_{(a_i,b_i)}^{\mathrm{split}}) = c_1 \frac{c_i}{c_1} \lambda_i$ (resp. $\rho_i^*(\boldsymbol{\lambda}_{(a_i,b_i)}^{\mathrm{inert}}) =  c_1 \frac{c_i}{c_1} \lambda_i$) for all $i$.

The result of the above observation is that, given an $n$-tuple $( (G_i, \iota_i, \lambda_i, \rho_i) )_{i=1}^n$, we may replace each quadruple with an isomorphic quadruple so that all quadruples satisfy  $\rho_i^*(\boldsymbol{\lambda}_{(a_i,b_i)}^{\mathrm{split}}) = c \lambda_i$ (resp. $\rho_i^*(\boldsymbol{\lambda}_{(a_i,b_i)}^{\mathrm{inert}}) = c \lambda_i$) for a single, fixed $c \in \Q_p^\times$ with $\ord_p(c) = j$. 

After making that adjustment,
$$\Psi( ( (G_i, \iota_i, \lambda_i, \rho_i) )_{i=1}^n )$$
is defined to be $(G, \iota, \lambda, \rho)$, where $G$ is the product (over $S$) of the $G_i$, the action $\iota$ is the product of the actions $\iota_i$, the polarization $\lambda$ is the product of the $\lambda_i$, and quasi-isogeny $\rho$ is the product of the $\rho_i$. 

Then $(G, \iota, \lambda, \rho)$ satisfies the local signature and Rosati involution conditions directly because the $(G_i, \iota_i, \lambda_i, \rho_i)$ satisfy their respective local signature and Rosati involution conditions, and $G$ is supersingular of height $2mn$ and dimension $mn$, since each $G_i$ is supersingular of height $2m$ and dimension $m$. Since we have $\rho_i^*( \boldsymbol{\lambda}_{(a_i,b_i)}^{\mathrm{split}})$ (resp. $\rho_i^*( \boldsymbol{\lambda}_{(a_i,b_i)}^{\mathrm{inert}}$)) equal to $c \lambda_i$ for all $i$, we have that $\rho^*(\lambda) = c \blambda$  for the product quasi-isogeny $\rho$ and the product polarization $\lambda$.

The final observation is that 
$$\Psi \circ \Phi ( G, \iota, \lambda, \rho) = (G, \iota, \lambda, \rho)$$
while 
$$\Phi \circ \Psi ( (G_i, \iota_i, \lambda_i, \rho_i) )_{i=1}^n = ( (G_i', \iota_i', \lambda_i', \rho_i') )_{i=1}^n$$
where each $(G_i, \iota_i, \lambda_i, \rho_i) \cong (G_i', \iota_i', \lambda_i', \rho_i')$. So, both compositions are the identity on isomorphism classes.

 \end{proof}

\section{Survey of Previous Results}

In the previous section, we showed that our unitary Rapoport-Zink space of local signature $( (a_i, b_i) )_{i=1}^n$ may be described in terms of Rapoport-Zink spaces parametrizing $p$-divisible groups with either an action of the ``split" quadratic extension $\Z_p \times \Z_p$ over $\Z_p$ or an action of the unramified quadratic extension $\mathcal{O}_L$ over $\Z_p$. In this section, we will record some previous results regarding the geometry of these Rapoport-Zink spaces for such quadratic extensions. 

One technical point is that given a quadruple $(G, \iota, \lambda, \rho)$ in any of these moduli spaces, there are two possible invariants of the isomorphism class of that quadruple: the height of the quasi-isogeny $\rho$, or the $\ord_p(c(\rho))$ for the element $c(\rho) \in \Q_p^\times$ comparing the polarization $\lambda$ and the pull-back of the basepoint polarization by $\rho$. We have chosen to focus on $\ord_p(c(\rho))$, while some important results are written in terms of the height of $\rho$. So first we will recall how to convert between these two invariants, as shown in Vollaard ~\cite{V} Lemma 1.7. The result is paraphrased in our notation below:

\begin{lemma}[Vollaard]\label{convert}  
Let $(G, \rho, \iota, \lambda)$ be a $k$-point of either $\widetilde{\mathcal{N}}^{inert,j}_{(p,q)}$ or $\widetilde{\mathcal{N}}^{split,j}_{(p,q)}.$ Then the height of $\rho$ is $mj$.
\end{lemma}
\begin{proof}
In ~\cite{V}, Vollaard proves this for $(G, \rho, \iota, \lambda) \in \widetilde{\mathcal{N}}^{inert,j}_{(p,q)}(k)$, but the proof only uses that 
$$\mathcal{O}_L \otimes_{\Z_p} W \cong W \times W.$$
Using the fact that
$$(\Z_p \times \Z_p) \otimes_{\Z_p} W \cong W \times W,$$
the exact same proof holds for $(G, \rho, \iota, \lambda) \in \widetilde{\mathcal{N}}^{split,j}_{(p,q)}(k)$.
\end{proof}

In particular, the open and closed formal subscheme of one of these unitary Rapoport-Zink spaces where $\ord_p(c(\rho)) = j$ may be identified with the open and closed formal subscheme where $\mathrm{height}(\rho) = mj$.

\begin{proposition}\label{empty}
If the functor $\widetilde{\mathcal{N}}^{\mathrm{split}}_{(p,q)}$ is non-empty, we must have $p=q$. 
\end{proposition}
\begin{proof}
Assume that the functor $\widetilde{\mathcal{N}}^{\mathrm{split}}_{(p,q)}$ is nonempty. Then there must be a basepoint triple $(\boldsymbol{G}^{\mathrm{split}}_{(p,q)},  \biota^{\mathrm{split}}_{(p,q)}, \boldsymbol{\lambda}_{(p,q)}^{\mathrm{split}})$ over $k$, satisfying the $(p,q)$ signature condition and the Rosati involution condition.  Because $\boldsymbol{G}^{\mathrm{split}}_{(p,q)}$ has an action of $\Z_p \times \Z_p$, we have an isomorphism of $p$-divisible groups over $k$:
$$\boldsymbol{G}^{\mathrm{split}}_{(p,q)} \cong \boldsymbol{G}_1 \times \boldsymbol{G}_2$$
where $\Z_p \times \Z_p$ acts through the first $\Z_p$-factor on $\boldsymbol{G}_1$ and the second on $\boldsymbol{G}_2$. 

Because $\boldsymbol{G}^{\mathrm{split}}_{(p,q)}$ satisfies the signature $(p,q)$ condition, we must have that $\boldsymbol{G}_1$ has dimension $p$ and that $\boldsymbol{G}_2$ has dimension $q$. Since $\boldsymbol{G}^{\mathrm{split}}_{(p,q)}$ is supersingular, both $\boldsymbol{G}_1$ and $\boldsymbol{G}_2$ are supersingular, so they are of height $2p$ and $2q$, respectively.

We have assumed that $\boldsymbol{\lambda}_{(p,q)}^{\mathrm{split}}$ is a principal polarization and that $\biota^{\mathrm{split}}_{(p,q)}$ and $\boldsymbol{\lambda}_{(p,q)}^{\mathrm{split}}$ satisfy the Rosati involution condition, so $\boldsymbol{\lambda}_{(p,q)}^{\mathrm{split}}$ defines isomorphisms:
$$\boldsymbol{\lambda}_{(p,q)}^{\mathrm{split}}: \boldsymbol{G}_1 \rightarrow (\boldsymbol{G}_2)^\vee \quad \quad \boldsymbol{\lambda}_{(p,q)}^{\mathrm{split}}: \boldsymbol{G}_2 \rightarrow (\boldsymbol{G}_1)^\vee.$$

In particular, 
$$\mathrm{height}( \boldsymbol{G}_1 ) = \mathrm{height}( (\boldsymbol{G}_1)^\vee ) = \mathrm{height}( \boldsymbol{G}_2 )$$
and so we have that $p=q$. 
\end{proof}
Note that by the above proposition, if $\widetilde{\mathcal{N}}^{\mathrm{split}}_{(p,q)}$ is nonempty then $m = p + q$ must be even. 

\begin{proposition}\label{11split}
The reduced $k$-scheme $\mathcal{N}^{split,j}_{(1,1)}$ is equi-dimensional of dimension 0.
\end{proposition}

There are  several ways to see that $\mathcal{N}^{split,j}_{(1,1)}$ is zero-dimensional. For example, by an argument analogous to  of ~\cite{GL4}, one can see that $\mathcal{N}^{split,j}_{(1,1)}$ is isomorphic to the underlying reduced scheme of the $\mathrm{GL}_2$ Rapoport-Zink space (the moduli space of supersingular $p$-divisible groups of height 1 and dimension 2, with a quasi-isogeny to a fixed basepoint), which is classically known to be zero-dimensional.

The following description of $\mathcal{N}^{split,j}_{(2,2)}$ is paraphrased from Theorem 7.3 of ~\cite{GL4}.

\begin{theorem}\label{22split} 
The reduced $k$-scheme $\mathcal{N}^{split,j}_{(2,2)}$ is equi-dimensional of dimension 1. Each irreducible component of $\mathcal{N}^{split,j}_{(2,2)}$ is isomorphic over $k$ to $\mathbb{P}^1$, and the irreducible components intersect as follows: there are $p^2 + 1$ intersection points on each irreducible component, and $p^2 + 1$ distinct irreducible components pass through each intersection point. 

As any two irreducible components intersect trivially or intersect in a single point, for a fixed irreducible component $X$  there are $p^2(p^2 +1)$ distinct irreducible components $X'$ such that 
$$X \cap X'$$
is a single point.

\end{theorem}

 \begin{proposition}\label{0minert}
If $mj$ is odd, the reduced $k$-scheme $\mathcal{N}^{inert,j}_{(0,m)}$ is empty. If $mj$ is even, $\mathcal{N}^{inert,j}_{(0,m)}$ is equi-dimensional of dimension 0.
\end{proposition}

\begin{proof}
By Kudla-Rapoport, $\widetilde{\mathcal{N}}^{inert}_{(0,m)}$ is formally smooth of dimension 0, in the sense that the completed local ring at any $k$-point of $\widetilde{\mathcal{N}}^{inert}_{(0,m)}$ is isomorphic to $W$. It follows (for example, from ~\cite{formal}) that the underlying reduced scheme $\mathcal{N}^{inert}_{(0,m)}$ is zero-dimensional as a $k$-scheme. In ~\cite{V}, Vollaard shows that the open and closed subschemes $\mathcal{N}^{inert,j}_{(0,m)}$ are only nonempty when $mj$ is even. So, the reduced $k$-schemes $\mathcal{N}^{inert,j}_{(0,m)}$ are zero-dimensional when $mj$ is even and are empty otherwise.

\end{proof}

The following description of $\mathcal{N}^{inert,j}_{(1,1)}$ follows from Corollary C of Vollaard-Wedhorn ~\cite{VW}: 

 \begin{proposition}\label{11inert}
For all $j$, the reduced $k$-scheme  $\mathcal{N}^{inert,j}_{(1,1)}$ is equi-dimensional of dimension 0.
\end{proposition}

The following description of $\mathcal{N}^{inert,j}_{(1,2)}$ is given by Vollaard in Theorem 5.16 of ~\cite{V}. After using Lemma \ref{convert} to present this theorem in our notation, the result is paraphrased below:

\begin{theorem}[Vollaard]\label{12inert} 
For $j$ odd, the reduced $k$-scheme $\mathcal{N}^{inert,j}_{(1,2)}$ is empty. For $j$ even, the reduced $k$-scheme $\mathcal{N}^{inert,j}_{(1,2)}$ is equi-dimensional of dimension 1. Each irreducible component of $\mathcal{N}^{inert,j}_{(1,2)}$ is isomorphic over $k$ to the Fermat curve:
$$C \ : \ x_0^{p+1} + x_1^{p+1} + x_2^{p+1} = 0 \subset \mathbb{P}^2_k.$$
There are $p^3 + 1$ intersection points on each irreducible component, and  $p + 1$ distinct irreducible components pass through each intersection point. 

As any two irreducible components intersect trivially or intersect in a single point, for a fixed irreducible component $X$  there are $p(p^3 + 1)$ distinct irreducible components $X'$ such that 
$$X \cap X'$$
is a single point.
\end{theorem}

The following description of $\mathcal{N}^{inert,j}_{(1,3)}$ is given by Vollaard-Wedhorn in Example 4.8  of~\cite{VW}. After using Lemma \ref{convert} to convert their statement into our notation, their result is paraphrased below:
\begin{theorem}[Vollaard-Wedhorn]\label{13inert} \
The reduced $k$-scheme $\mathcal{N}^{inert,j}_{(1,3)}$ is equi-dimensional of dimension 1. Each irreducible component of $\mathcal{N}^{inert,j}_{(1,3)}$ is isomorphic over $k$ to the Fermat curve:
$$C \ : \ x_0^{p+1} + x_1^{p+1} + x_2^{p+1} = 0 \subset \mathbb{P}^2_k.$$
There are $p^3 + 1$ intersection points on each irreducible component, and  $p^3 + 1$ distinct irreducible components pass through each intersection point. 

As any two irreducible components intersect trivially or intersect in a single point, for a fixed irreducible component $X$  there are $p^3(p^3 + 1)$ distinct irreducible components $X'$ such that 
$$X \cap X'$$
is a single point.
\end{theorem}

The following description of $\mathcal{N}^{inert,j}_{(2,2)}$ is given in Theorem 3.10 and Theorem 3.12  of Howard-Pappas ~\cite{GU22}, and is paraphrased here in our notation:
\begin{theorem}[Howard-Pappas]\label{22inert} 
The reduced $k$-scheme $\mathcal{N}^{inert,j}_{(2,2)}$ is equi-dimensional of dimension 2.  Each irreducible component of $\mathcal{N}^{inert,j}_{(2,2)}$ is isomorphic over $k$ to the Fermat surface:
$$S : \  x_0^{p+1} + x_1^{p+1} + x_2^{p+1} + x_3^{p+1} = 0 \subset \mathbb{P}^3_k.$$
For a fixed irreducible component $X$  there are $(p^3 + 1)(p^2 + 1)$ distinct irreducible components $X'$ such that 
$$X \cap X'$$
is a single point and $(p^3 + 1)(p+1)$ distinct irreducible components $X'$ such that 
$$X \cap X'$$
is a projective line.
\end{theorem}

\section{Application to Rapoport-Zink Spaces}

Recall that in a collection of local signatures  $( (a_i, b_i) )_{i=1}^n$, there is a fixed integer $m$ such that $a_i + b_i = m$ for all $i$. The results surveyed in the previous section can be combined with Theorem ~\ref{Decomp} to describe the geometry of the Rapoport-Zink spaces $\mathcal{N}_{( (a_i, b_i) )_{i=1}^n}$, under the assumption that $1 \leq m \leq 4$. The complexity increases with $m$, so we will begin with $m=1$:

\begin{corollary}\label{RZ1}
Assume that $m=1$.  If all $\mathfrak{p}_i$ are inert in $E$, the reduced $k$-scheme  $\mathcal{N}_{( (a_i, b_i) )_{i=1}^n}$ is zero-dimensional. The components  $\mathcal{N}^j_{( (a_i, b_i) )_{i=1}^n}$  are empty for $j$ odd and zero-dimensional for $j$ even. If any of the primes $\mathfrak{p}_i$  are split in $E$, then $\mathcal{N}_{( (a_i, b_i) )_{i=1}^n}$ is empty.

\end{corollary}

\begin{proof}
Assume that $m=1$, and note that then all $(a_i, b_i)$ are equal to $(0,1)$.  Recall our notation that $\mathfrak{p}_i$ is split in $E$ for all $i \leq h$ and inert in $E$ for $i>h$ (in particular, $h$ might be 0 or $n$). By Theorem ~\ref{Decomp}, for any  integer $j$,
$$\mathcal{N}^{j}_{( (0, 1) )_{i=1}^n} \cong  (\mathcal{N}^{\mathrm{split},j}_{(0,1)} )^h \times_k  ( \mathcal{N}^{\mathrm{inert},j}_{(0,1)} )^{n-h}.$$ 

As $\mathcal{N}^{\mathrm{split},j}_{(0,1)}$ is empty for all $j$ and $\mathcal{N}^{\mathrm{inert},j}_{(0,1)}$ is empty for odd $j$, if either $h > 0$ or $j$ is odd, $\mathcal{N}^{j}_{( (0, 1) )_{i=1}^n}$ is empty. Otherwise, $j$ is even and $\mathcal{N}^{j}_{( (0, 1) )_{i=1}^n}$ is a product of $n$ copies of $\mathcal{N}^{\mathrm{inert},j}_{(0,1)}$, which are all zero-dimensional.
\end{proof}

\begin{corollary}\label{RZ2}
Assume that $m=2$.  Let $c$ be the number of $(a_i, b_i) =  (0,2)$ in the collection of local signatures $(( a_i, b_i))_{i=1}^n$ such that $\mathfrak{p}_i$ is split in $E$. If $c= 0$, the reduced $k$-scheme $\mathcal{N}_{( (a_i, b_i) )_{i=1}^n}$ is zero-dimensional.  Otherwise, $\mathcal{N}_{( (a_i, b_i) )_{i=1}^n}$ is empty.
\end{corollary}

\begin{proof}
Assume that $m=2$. For this proof only, we will use the following notation based on the local signature $(( a_i, b_i))_{i=1}^n$  condition used in defining $\mathcal{N}_{( (a_i, b_i) )_{i=1}^n}$:
$$c =  \ \text{ number of } \  (a_i, b_i) = (0,2)  \ \text{ such that } \ \mathfrak{p_i} \ \text{ is split in } \ E$$
$$d =  \ \text{ number of } \  (a_i, b_i) = (1,1)  \ \text{ such that } \ \mathfrak{p_i} \ \text{ is split in } \ E$$
$$e =  \ \text{ number of } \  (a_i, b_i) = (0,2)  \ \text{ such that } \ \mathfrak{p_i} \ \text{ is inert in } \ E$$
$$f =  \ \text{ number of } \  (a_i, b_i) = (1,1)  \ \text{ such that } \ \mathfrak{p_i} \ \text{ is inert in } \ E$$
By Theorem ~\ref{Decomp}, for any  integer $j$,
$$\mathcal{N}_{( (a_i, b_i) )_{i=1}^n}^{j} \cong   \big( \mathcal{N}^{\mathrm{split}, j}_{(0,2)} \big)^{c} \times_k  \big( \mathcal{N}^{\mathrm{split}, j}_{(1,1)} \big)^{d}  \times_k   \big( \mathcal{N}^{\mathrm{inert}, j}_{(0,2)} \big)^{e}  \times_k \big( \mathcal{N}^{\mathrm{inert}, j}_{(1,1)} \big)^{f} $$ 
The reduced $k$-schemes $\mathcal{N}^{\mathrm{split}, j}_{(1,1)}$, $\mathcal{N}^{\mathrm{inert}, j}_{(0,2)}$, and $\mathcal{N}^{\mathrm{inert}, j}_{(1,1)}$ are all zero-dimensional, while $\mathcal{N}^{\mathrm{split}, j}_{(0,2)}$ is empty. So, $\mathcal{N}_{( (a_i, b_i) )_{i=1}^n}^{j}$ is zero-dimensional whenever $c = 0$ and is empty otherwise.
\end{proof}

The next two cases, $m=3$ and $m=4$, are similar. For the sake of brevity, we will begin with the more complicated case of $m=4$, and then simply explain how the case $m=3$ differs from this.

\begin{corollary}\label{RZ}
Assume that $m=4$. Define the following integers based on the local signature condition ${( (a_i, b_i) )_{i=1}^n}$ for $\mathcal{N}_{( (a_i, b_i) )_{i=1}^n}$:
$$c =  \ \text{ number of } \  (a_i, b_i) = (0,4)  \ \text{ such that } \ \mathfrak{p_i} \ \text{ is inert in } \ E$$
$$d =  \ \text{ number of } \  (a_i, b_i) = (1,3)  \ \text{ such that } \ \mathfrak{p_i} \ \text{ is inert in } \ E$$
$$e =  \ \text{ number of } \  (a_i, b_i) = (2,2)  \ \text{ such that } \ \mathfrak{p_i} \ \text{ is inert in } \ E$$
$$f =  \ \text{ number of } \  (a_i, b_i) = (2,2)  \ \text{ such that } \ \mathfrak{p_i} \ \text{ is split in } \ E$$
$$g = \ \text{ number of } \  (a_i, b_i) \neq (2,2)  \ \text{ such that } \ \mathfrak{p_i} \ \text{ is split in } \ E$$
Note that any of these constants might be zero, and $c+d+e+f+g= n$. If $g \neq 0$, then $\mathcal{N}_{( (a_i, b_i) )_{i=1}^n}$ is empty. If $g = 0$,  then each $\mathcal{N}_{( (a_i, b_i) )_{i=1}^n}^j$ is nonempty and has geometry as follows:\\

\begin{itemize}
\item{$\mathcal{N}_{( (a_i, b_i) )_{i=1}^n}^j$ is equi-dimensional of dimension $d+2e+f$.\\}

\item{Each irreducible component of $\mathcal{N}_{( (a_i, b_i) )_{i=1}^n}^j$ is isomorphic over $k$ to:
$$  C^d \times_k S^{e} \times_k (\mathbb{P}^1)^{f},$$
 a product (over $k$) of $b$ copies of the Fermat curve, $e$ copies of the Fermat surface, and $f$ copies of the projective line. \\}
 
 \item{Any two irreducible components either intersect trivially or have intersection isomorphic over $k$ to:
 $$ C^{r} \times_k S^{s_1} \times_k (\mathbb{P}^1)^{s_2 +t}$$
 for non-negative integers $r, s_1, s_2, t$ satisfying the following three conditions:
\begin{enumerate}
\item{ $$0 \leq r \leq d$$ }
\item{ $$0 \leq s_1 + s_2 \leq e$$ }
\item{ $$0 \leq t \leq f.$$ }
\end{enumerate} }

\item{ For a fixed irreducible component $X$ of  $\mathcal{N}_{( (a_i, b_i) )_{i=1}^n}^j$, there are
 $$\big( p^3(p^3 + 1) \big)^{d-r}   \big(   (p^3 + 1)(p^2+1)  \big)^{e-s_1-s_2}  \big(   (p^3 + 1)(p+1) \big)^{s_2}  \big(  p^2(p^2 + 1) \big)^{f-t}$$   
 irreducible components $X'$ such that $X \cap X'$ is isomorphic over $k$ to:
 $$ C^{r} \times_k S^{s_1} \times_k (\mathbb{P}^1)^{s_2 +t}$$}
  \end{itemize}
\end{corollary}

\begin{proof}
To see that $\mathcal{N}_{( (a_i, b_i) )_{i=1}^n}$ is empty whenever $g \neq 0$, combine Theorem ~\ref{Decomp} with Proposition \ref{empty}. For the rest of this proof, assume that $g = 0$. As $\mathcal{N}_{( (a_i, b_i) )_{i=1}^n} = \bigsqcup_{j \in \Z} \mathcal{N}_{( (a_i, b_i) )_{i=1}^n}^j$, it is enough to understand the geometry of each $\mathcal{N}_{( (a_i, b_i) )_{i=1}^n}^j$.   Applying Theorem ~\ref{Decomp} in terms of the integers $c$, $d$, $e$, and $f$ defined above, we have:

$$\mathcal{N}_{( (a_i, b_i) )_{i=1}^n}^{j} \cong   \big( \mathcal{N}^{\mathrm{inert}, j}_{(0,4)} \big)^{c} \times_k  \big( \mathcal{N}^{\mathrm{inert}, j}_{(1,3)} \big)^{d}  \times_k   \big( \mathcal{N}^{\mathrm{inert}, j}_{(2,2)} \big)^{e}  \times_k \big( \mathcal{N}^{\mathrm{split}, j}_{(2,2)} \big)^{f} $$ 

By  Proposition ~\ref{0minert}, Theorem ~\ref{13inert}, Theorem ~\ref{22inert}, and Theorem ~\ref{22split}, respectively, we have that 
\begin{itemize}
\item{$\mathcal{N}^{\mathrm{inert}, j}_{(0,4)}$  is zero-dimensional} 
\item{Every irreducible component of $\mathcal{N}^{\mathrm{inert}, j}_{(1,3)}$ is isomorphic to the Fermat curve $C$}
\item{Every irreducible component of $\mathcal{N}^{\mathrm{inert}, j}_{(2,2)}$ is isomorphic to the Fermat surface $S$}
\item{Every irreducible component of $ \mathcal{N}^{\mathrm{split}, j}_{(2,2)}$ is isomorphic to the projective line $\mathbb{P}^1$.}
 \end{itemize}

It follows that every irreducible component of $\mathcal{N}_{( (a_i, b_i) )_{i=1}^n}^{j}$ is isomorphic to
$$  C^d \times_k S^{e} \times_k (\mathbb{P}^1)^{f},$$
and that $\mathcal{N}_{( (a_i, b_i) )_{i=1}^n}^{j}$ is equi-dimensional of dimension $d+2e+f$.

Also recall that by Theorem ~\ref{13inert}, Theorem ~\ref{22inert}, and Theorem ~\ref{22split}, respectively,
\begin{itemize}
\item{For a fixed irreducible component $X \cong C$  of $\mathcal{N}^{inert, j}_{(1,3)}$ there are $p^3(p^3 + 1)$ distinct irreducible components $X' \cong C$ such that $X \cap X'$ is a single point.}
\item{For a fixed irreducible component $X \cong S$ of $\mathcal{N}^{inert, j}_{(2,2)}$ there are $(p^3 + 1)(p^2 + 1)$ distinct irreducible components $X' \cong S$ such that $X \cap X'$ is a single point and $(p^3 + 1)(p+1)$ distinct irreducible components $X' \cong S$ such that $X \cap X'$ is a projective line.}
\item{For a fixed irreducible component $X \cong \mathbb{P}^1$ of $\mathcal{N}^{split, j}_{(2,2)}$  there are $p^2(p^2 +1)$ distinct irreducible components $X' \cong \mathbb{P}^1$ such that $X \cap X'$ is a single point. }
\end{itemize}

For a fixed irreducible component 
$$X =  C_1 \times_k \cdots \times_k  C_d  \times_k S_1 \times_k \cdots \times_k  S_e \times_k \mathbb{P}_1^1 \times_k \cdots \times_k  \mathbb{P}_f^1$$
of $\mathcal{N}_{( (a_i, b_i) )_{i=1}^n}^{j}$, another irreducible component  
$$X' =   C'_1 \times_k \cdots \times_k  C'_d  \times_k S'_1 \times_k \cdots \times_k  S'_e \times_k (\mathbb{P}_1^1)' \times_k \cdots \times_k  (\mathbb{P}_f^1)'$$
will intersect nontrivially with $X$ if and only if each $C_i$ (respectively $S_i,   \ \mathbb{P}_i$) is either equal to $C'_i$ (respectively $S_i, \ ( \mathbb{P}_i)'$) or intersects $C'_i$ (respectively $S'_i,  \ (\mathbb{P}_i)'$)  nontrivially, when viewed as an irreducible component of $\mathcal{N}^{\mathrm{inert}, j}_{(1,3)}$ (respectively $\mathcal{N}^{\mathrm{inert}, j}_{(2,2)}, \ \mathcal{N}^{\mathrm{split}, j}_{(2,2)}$) via the isomorphism $\Phi$ described in Theorem ~\ref{Decomp}.

By letting $r$ to be the number of $C'_i$ equal to $C_i$ (so $d-r$ of the $C'_i$ intersect $C_i$ in a single point), letting $s_1$ be the number of $S'_i$ equal to $S_i$ and $s_2$ be the number of $S'_i$ intersecting $S_i$ in a projective line (so $e - s_1 - s_2$ of the $S'_i$ intersect $S_i$ in a single point), and letting $t$ be the number of $(\mathbb{P}^1_i)'$ equal to $\mathbb{P}^1_i$ (so $f-t$ of the $(\mathbb{P}^1_i)'$ intersect $\mathbb{P}^1_i$ in a single point), we have that there are 
$$\big( p^3(p^3 + 1) \big)^{d-r}   \big(   (p^3 + 1)(p^2+1)  \big)^{e-s_1-s_2}  \big(   (p^3 + 1)(p+1) \big)^{s_2}  \big(  p^2(p^2 + 1) \big)^{f-t}$$   
 irreducible components $X'$ such that $X \cap X'$ is isomorphic over $k$ to:
  $$ C^r\times_k S^{s_1} \times_k (\mathbb{P}^1)^{s_2 +t}.$$
\end{proof}

\begin{corollary}\label{RZ3}
Assume that $m=3$.  If any of the  $\mathfrak{p}_i$ are split in $E$, the moduli space $\mathcal{N}_{( (a_i, b_i) )_{i=1}^n}$ is empty. Otherwise, let $c$ be the number of $(a_i, b_i) = (1,2)$. Then the reduced $k$-scheme $\mathcal{N}_{( (a_i, b_i) )_{i=1}^n}$  has geometry as follows:
\begin{itemize}
\item{For $j$ even, $\mathcal{N}_{( (a_i, b_i) )_{i=1}^n}^j$  is equi-dimensional of dimension $c$. For $j$ odd, $\mathcal{N}_{( (a_i, b_i) )_{i=1}^n}^j$  is empty.}
\item{Each irreducible component of $\mathcal{N}_{( (a_i, b_i) )_{i=1}^n}^j$ is isomorphic over $k$ to $C^{c}$,
 a product (over $k$) of $c$ copies of the Fermat curve $C$.}
 \item{Any two irreducible components either intersect trivially or have intersection isomorphic over k to $C^{t}$
 for some integer $0 \leq t \leq c$.}
 \item{For a fixed irreducible component $X$ of $\mathcal{N}_{( (a_i, b_i) )_{i=1}^n}^j$, there are $( p(p^3 + 1) )^{c-t}$ irreducible components $X'$ such that $X \cap X'$ is isomorphic over $k$ to $C^{t}$.}
\end{itemize}
\end{corollary}
\begin{proof}
The proof of this statement is analogous to the proof of Corollary \ref{RZ}. In this case, $\mathcal{N}^j_{( (a_i, b_i) )_{i=1}^n}$ is isomorphic to some number of copies of $\mathcal{N}^{split, j}_{(0,3)}$, $\mathcal{N}^{split, j}_{(1,2)}$, $\mathcal{N}^{inert, j}_{(0,3)}$, and $\mathcal{N}^{inert, j}_{(1,2)}$. The fact that  $\mathcal{N}_{( (a_i, b_i) )_{i=1}^n}$ is empty whenever any of the primes $\mathfrak{p}_i$ are split in $E$ comes from the fact that $\mathcal{N}^{split, j}_{(0,3)}$ and $\mathcal{N}^{split, j}_{(1,2)}$ are both empty. 

When none of the $\mathfrak{p}_i$ are split in $E$, we have that
$$\mathcal{N}^j_{( (a_i, b_i) )_{i=1}^n} \cong (\mathcal{N}^{inert, j}_{(1,2)})^c \times_k  (\mathcal{N}^{inert, j}_{(0,3)})^{n-c},$$
where $c$  is the number of the number of $(a_i, b_i) = (1,2)$. Note that for $j$ odd, both $\mathcal{N}^{inert, j}_{(1,2)}$ and $\mathcal{N}^{inert, j}_{(0,3)}$ are empty. For $j$ even, $\mathcal{N}^{inert, j}_{(0,3)}$ is zero-dimensional, so the geometry of $\mathcal{N}^j_{( (a_i, b_i) )_{i=1}^n}$ is determined by the geometry of $\mathcal{N}^{inert, j}_{(1,2)}$.

In this case of $m=3$, the intersection combinatorics for $\mathcal{N}^j_{( (a_i, b_i) )_{i=1}^n}$ are determined by the intersection combinatorics for $\mathcal{N}^{inert, j}_{(1,2)}$ in the same way that for $m=4$, the intersection combinatorics for $\mathcal{N}_{( (a_i, b_i) )_{i=1}^n}$ are determined by those of $ \mathcal{N}^{\mathrm{inert}, j}_{(1,3)}$, $ \mathcal{N}^{\mathrm{inert}, j}_{(2,2)}$, and $ \mathcal{N}^{\mathrm{split}, j}_{(2,2)}$. 
\end{proof}

\section{Application to Supersingular Loci}\label{appl}

We will now describe the supersingular locus at a prime $\nu$ of the reflex field $F$ of the unitary Shimura variety $\mathrm{Sh}^V_K$ introduced in Section \ref{intro}. We must first describe the integral model $\mathcal{M}_K$ discussed in the introduction more carefully.

Recall that $\mathcal{O}$ is the integral closure of $\Z_{(p)}$ in $E$, the nontrivial automorphism of $E$ over $E^+$ is denoted $e \mapsto \overline{e}$, and that we've fixed a free $\mathcal{O}$-module of rank $m$ with a perfect $\mathcal{O}$-valued hermitian form $\langle \cdot, \cdot \rangle$. The  $n$ embeddings of $E^+$ into $\overline{\Q}$ are denoted by $\{ \varphi_{i} \}_{i=1}^n$ and the $2n$ embeddings of $E$ into $\overline{\Q}$ are denoted by $\{ \varphi_{i,a} , \varphi_{i,b} \}_{i=1}^n$. Let  $(a_i, b_i)$ be the signature of $V \otimes_{\mathcal{O}} E$ at the place defined by $\varphi_i$.

Let $G = \GU(V \otimes_{\mathcal{O}} E)$, viewed as an algebraic group over $\Q$. Note that:
$$G(\mathbb{R}) \cong G\left( \prod_{i=1}^n \mathrm{U}(a_i, b_i) \right)$$
where $\mathrm{U}(a_i,b_i)$ is the real unitary group of signature $(a_i, b_i)$, and the notation $G( \ \ )$ refers to the fact that the similitude factors for all $i$ must be the same.
 
 Let $h$ be the homomorphism of real algebraic groups
 $$h: \ \mathrm{Res}^\C_{\R} (\mathbb{G}_m ) \rightarrow G_{\R}$$
 that sends $z \in \C^\times$ to the matrix $\mathrm{diag}(z^{a_1}, \overline{z}^{b_1}, \dots, z^{a_n}, \overline{z}^{b_n} )$. Then the tuple $(E, \mathcal{O}, e \mapsto \overline{e}, V, \langle \cdot, \cdot \rangle, G, h)$ is a PEL-datum, in the sense of Kottwitz \cite{Ko}, with reflex field $F$.

Recall that $K_p = GU(V)(\Z_p)$, $K^p \subset GU(V)(\mathbb{A}_p^f)$ is a compact open subgroup, and $K = K_pK^p$. For $K^p$ sufficiently small, and $mn$ even the Shimura variety $\mathrm{Sh}^K_V$ admits an integral canonical model $\mathcal{M}_K$ over $\mathcal{O}_F \otimes_\Z \Z_{(p)}$, as defined by Kottwitz ~\cite{Ko}, based on the PEL-datum above. For $mn$ odd, the complex fiber of $\mathcal{M}_K$ is a disjoint union of copies of $\mathrm{Sh}^K_V$. For a scheme $S$ over $\mathcal{O}_F \otimes_\Z \Z_{(p)}$, the moduli space $\mathcal{M}_K$ parametrizes prime-to-$p$ isogeny classes of quadruples $(A, \iota, \lambda, \eta^pK^p)$, where:

\begin{itemize}
\item{ $A$ is a projective abelian scheme over $S$ of relative dimension $mn$.}
\item{ $\iota: \mathcal{O} \rightarrow \mathrm{End}(A) \otimes_{\Z} \Z_{(p)}$ is an action of $\mathcal{O}$ on $A$, satisfying the \emph{global signature $((a_i, b_i))_{i=1}^n$ condition}: For all $e \in \mathcal{O}$,
$$\det(T - \iota(e); \mathrm{Lie}(A) ) = \prod_{i=1}^n(T - \varphi_{i,a}(e))^{a_i}(T - \varphi_{i,b}(e))^{b_i}.$$
Note that the coefficients of this polynomial lie in $\mathcal{O}_F \otimes_\Z \Z_{(p)}$.}
\item{$\lambda \in \Hom(A, A^\vee) \otimes_{\Z} \Z_{(p)}$ is a prime-to-$p$ quasi-polarization of $A$, which satisfies the \emph{Rosati involution condition}: for all $e \in \mathcal{O}$,
$$\lambda \circ \iota(\overline{e}) = \iota(e)^\vee \circ \lambda,$$
where $\iota(e)^\vee$ is the dual isogeny to $\iota(e)$.}
\item{The level structure $\eta^pK^p$ is the $K^p$ orbit of an $\mathcal{O} \otimes \A_f^p$-linear isomorphism:
$$\eta^p: H_1(A, \mathbb{A}_f^p) \rightarrow V \otimes \A_f^p$$
that respects the hermitian forms on either side, up to scaling by $(\A_f^p)^\times$.}
\end{itemize}

For a prime $\nu$ of $F$ over $p$, the supersingular locus $\mathcal{M}_{K,\nu}^{ss}$ of $\mathcal{M}_K \times_{    \mathcal{O}_{F} \otimes_{\Z} \Z_{(p)}, \nu  } k$ is uniformized by a Rapoport-Zink space of the form $\mathcal{N}_{( (a_{\nu,i}, b_{\nu,i}) )_{i=1}^n}$. To apply this, we need to examine the signature conditions more carefully:

We will give an equivalent formulation of the global signature $((a_i, b_i))_{i=1}^n$ condition (note that the following discussion uses the assumption that $p$ is unramified in $E$). The homomorphism $h$ determines a decomposition:
$$V_{\C} = V_0 \oplus V_1$$
where $V_0$ is the subspace of $V_{\C}$ on which $h(z)$ acts by $z$ and $V_1$ is the subspace of $V_{\C}$ on which $h(z)$ acts by $\overline{z}$ (this decomposition is actually defined over $F$).  As
$$\mathrm{det}(T - e; V_0) = \prod_{i=1}^n(T - \varphi_{i,a}(e))^{a_i}(T - \varphi_{i,b}(e))^{b_i}$$
we can rephrase the global signature $((a_i, b_i))_{i=1}^n$ condition as:
$$\det(T - \iota(e); \mathrm{Lie}(A) ) = \mathrm{det}(T - e; V_0)$$
(this is the determinant condition as defined by Kottwitz).

Given such a PEL-datum and a choice of prime $\nu$ of the reflex field $F$ over $p$, Rapoport and Zink construct a Rapoport-Zink space $\widetilde{\mathcal{N}}$ with underlying reduced scheme $\mathcal{N}$ uniformizing the supersingular locus $\mathcal{M}_{K,\nu}^{ss}$ (see Rapoport-Zink \cite{RZ} Chapter 6). Fix a prime $\nu$ of $F$. The Rapoport-Zink space $\widetilde{\mathcal{N}}$ needed to uniformize $\mathcal{M}_{K,\nu}^{ss}$ is one of the moduli spaces $\widetilde{\mathcal{N} }_{( (a_{\nu,i}, b_{\nu,i}) )_{i=1}^n}$ defined in Section ~\ref{RZdef}, where the local signature $( (a_{\nu, i}, b_{\nu, i}) )_{i=1}^n$ condition is derived from $V$ and $\nu$ as follows: 

As the decomposition above is defined over a number field, we will fix a decomposition over $\overline{\Q}$:
$$V \otimes_{\Q} \overline{\Q} = V_0 \oplus V_1.$$
The prime $\nu$ gives an embedding $F \rightarrow \overline{\Q}_p$, which can be extended to an embedding $\rho_v: \overline{\Q} \rightarrow \overline{\Q}_p$. Then, the splitting of $V \otimes_{\Q} \overline{\Q}$ above gives a splitting:
$$(V \otimes_{\Q} \overline{\Q}) \otimes_{\overline{\Q}, \rho_\nu} \overline{\Q}_p = V_0 \otimes_{\overline{\Q}, \rho_\nu} \overline{\Q}_p \oplus V_1 \otimes_{\overline{\Q}, \rho_\nu} \overline{\Q}_p.$$
(This actually occurs over a finite extension of $\Q_p$.)

The local signature condition at $\nu$ used in defining $\widetilde{\mathcal{N}}$ is that:
$$\det(T - \iota(e); \mathrm{Lie}(G) ) = \mathrm{det}(T - e; V_0  \otimes_{\overline{\Q}, \rho_{\nu} } \overline{\Q}_p  )$$
for all $e \in \mathcal{O}_E \otimes_{\Z} \Z_p$.  As
$$\mathrm{det}(T - e; V_0) = \prod_{i=1}^n(T - \varphi_{i,a}(e))^{a_i}(T - \varphi_{i,b}(e))^{b_i},$$
the local signature condition at $\nu$ is:
$$\det(T - \iota(e); \mathrm{Lie}(G) ) = \mathrm{det}(T - e; V_0  \otimes_{\overline{\Q}, \rho_{\nu} } \overline{\Q}_p  ) =  \prod_{i=1}^n(T - \rho_\nu \circ \varphi_{i,a}(e))^{a_i}(T - \rho_\nu \circ \varphi_{i,b}(e))^{b_i}.$$

In particular, from the above description of the signature condition, it is clear that this depends only on $V$ and $\nu$, and not the choice of $\rho_\nu$ extending $\nu: F \rightarrow \overline{\Q}_p$. As $\phi \mapsto \rho_\nu \circ \phi$ gives an isomorphism $\Hom_{\Q}(E, \overline{\Q}) \rightarrow \Hom_{\Q}(E, \overline{\Q}_p)$ (that does \emph{not} necessarily identify $\varphi_{i,a}$ with $\psi_{i,a}$), we have
$$\mathrm{det}(T - e; V_0  \otimes_{\overline{\Q}, \rho_{\nu} } \overline{\Q}_p  )  = \prod_{i=1}^n(T - \psi_{i,a}(e))^{a_{\nu,i}}(T - \psi_{i,b}(e))^{b_{\nu,i}},$$
which is the polynomial defining the local signature $( (a_{\nu, i}, b_{\nu,i}) )_{i=1}^n$ condition, for some integers $a_{\nu,i}, b_{\nu, i}$. We assume without loss of generality that $a_{\nu,i} \leq b_{\nu,i}$ for all $i$. 


\begin{definition}\label{localsign}
Given a prime $\nu$ of $F$, let $( (a_{\nu, i}, b_{\nu,i}) )_{i=1}^n$ denote the collection of local signatures constructed above from $V$ and $\nu$. This is the list of integers with the property that:
$$\mathrm{det}(T - e; V_0  \otimes_{\overline{\Q}, \rho_{\nu} } \overline{\Q}_p  ) = \prod_{i=1}^n(T - \psi_{i,a}(e))^{a_{\nu,i}}(T - \psi_{i,b}(e))^{b_{\nu,i}}.$$

\end{definition}

We can now use the description given in Theorem \ref{RZ} of the geometry of $\mathcal{N}_{( (a_{\nu,i}, b_{\nu,i}) )_{i=1}^n}$ to give a description of the geometry of $\mathcal{M}_{K,\nu}^{ss}$. Fix a point $(\boldsymbol{A}, \biota, \blambda, \boldsymbol{\eta^p}K^p) \in \mathcal{M}_{K,\nu}^{ss}(k)$. We will use this point to define a particular choice of basepoint of $\widetilde{\mathcal{N}}_{( (a_{\nu,i}, b_{\nu,i}) )_{i=1}^n}$, which we are free to do by  ~\cite{VW} Lemma 5.1.

 Let $\bG = \boldsymbol{A}[p^\infty]$ be the $p$-divisible group of $\boldsymbol{A}$. The action $\boldsymbol{\iota}: \mathcal{O} \rightarrow \mathrm{End}(A) \otimes_{\Z} \Z_{(p)}$ on $A$ induces an action, also denoted $\boldsymbol{\iota}$, of $\mathcal{O}_E \otimes_{\Z} \Z_p$ on $\boldsymbol{G}$. The prime-to-$p$ quasi-polarization $\blambda$ of $\boldsymbol{A}$ induces a principal polarization, also denoted $\blambda$, of $\boldsymbol{G}$. The fact the action $\biota$ on $\boldsymbol{A}$ satisfies the global $( (a_i, b_i) )_{i=1}^n$ signature condition implies that the induced action $\biota$ on $\boldsymbol{G}$ satisfies the local $( (a_{\nu,i}, b_{\nu,i} ) )_{i=1}^n$ signature condition. 

 Define $J$ as the subgroup of $(\End(\G) \otimes_{\Z} \Q)^\times$ of quasi-endomorphisms $g$ of $\bG$ that are $\mathcal{O}_E \otimes_{\Z} \Z_p$-linear and respect the polarization of $\bG$, in the sense that $g^*(\blambda) = c(g) \blambda$, for some $c(g) \in \Q_p^\times$. There is an algebraic group $I$ over $\Q$ such that $I(\Q)$ is the subgroup of $(\End(\boldsymbol{A}) \otimes_{\Z} \Q)^\times$ of quasi-endomorphisms $g$ of $\boldsymbol{A}$ that are $\mathcal{O}$ linear and respect the polarization of $\boldsymbol{A}$, in the sense that $g^*(\blambda) = c(g) \blambda$ for some $\nu(g) \in \Q^\times$.  For this algebraic group $I$, we have that $I(\Q_p) \cong J$. The level structure $\boldsymbol{\eta^p} K^p$ determines a right $K^p$-orbit of isomorphisms $I(\A_f^p) \cong G(\A_f^p)$. Using this, $I(\Q)$ acts on both $\mathcal{N}_{( (a_{\nu,i}, b_{\nu,i}) )_{i=1}^n}$ and $G(\A_f^p)/K^p$.  Now we may state the Rapoport-Zink uniformization theorem:

\begin{theorem}[Rapoport-Zink]
There is an isomorphism of $k$-schemes:
$$\mathcal{M}_{K,\nu}^{ss} \cong I(\Q) \setminus (\mathcal{N}_{( (a_{\nu,i}, b_{\nu,i}) )_{i=1}^n} \times G(\A_f^p) / K^p).$$
\end{theorem}

There are finitely many double cosets in $I(\Q) \setminus G(\A_f^p) / K^p$, so let $g_1, ..., g_r$ be representatives of these double cosets. Let $\Gamma_j = I(\Q) \cap g_j K^p g_j^{-1}$, for all $1 \leq j \leq r.$ Then,
$$\bigsqcup_{i=1}^r \Gamma_j \setminus \mathcal{N}_{( (a_{\nu,i}, b_{\nu,i}) )_{i=1}^n} \cong  I(\Q) \setminus (\mathcal{N}_{( (a_{\nu,i}, b_{\nu,i}) )_{i=1}^n} \times G(\A_f^p) / K^p).$$
\begin{lemma}\label{localisom}
For $K^p$ sufficiently small, the surjective morphism:
$$\mathcal{N}_{( (a_{\nu,i}, b_{\nu,i}) )_{i=1}^n} \rightarrow \bigsqcup_{i=1}^r \Gamma_j \setminus \mathcal{N}_{( (a_{\nu,i}, b_{\nu,i}) )_{i=1}^n} \cong \mathcal{M}_{K,\nu}^{ss}$$
induced by the isomorphism in the Rapoport-Zink Uniformization Theorem is a local isomorphism.
\end{lemma}
This was proven by Vollaard ~\cite{V} in the case that $E^+ = \Q$, the hermitian space $V$ is of  signature $(1,2)$, and the prime $p$ is inert. However, as mentioned in the papers of  Howard-Pappas ~\cite{GU22}, Vollaard-Wedhorn ~\cite{VW}, and the author ~\cite{GL4}, the analogous statement still holds for $E^+ = \Q$ when   the signature is equal to $(2,2)$ with $p$ inert, the signature is equal to $(1,3)$ with $p$ inert, or the signature is equal to $(2,2)$ with $p$ split, respectively. These results may be combined in the local setting to yield the above lemma for general $E^+$.

Using the Rapoport-Zink Uniformization Theorem and Lemma ~\ref{localisom}, the local results in Corollaries \ref{RZ1} through \ref{RZ3} implies the following description of the supersingular locus.

\begin{theorem}
Assume that $p$ splits completely in $E^+$ and is unramified in $E$. Fix a prime $\nu$ of the reflex field $F$ over $p$ and a sufficiently small compact open subgroup $K^p \subset G(\mathbb{A}_f^p)$. Let $(( a_i, b_i))_{i=1}^n$ be the tuple  defining the global signature condition for $\mathcal{M}_{K}$, and let  $( (a_{\nu, i}, b_{\nu,i}) )_{i=1}^n$ be the tuple of local signatures determined by $\nu$ as in Definition \ref{localsign}.  We will describe the geometry of the supersingular locus $\mathcal{M}_{K,\nu}^{ss}$ in the four cases $m=1$, $2$, $3$, and $4$. \\

\textbf{\underline{$\boldsymbol{m=1}:$} }\\
 If all $\mathfrak{p}_i$ are inert in $E$, the reduced $k$-scheme  $\mathcal{M}_{K,\nu}^{ss}$ is zero-dimensional.  If any of the primes $\mathfrak{p}_i$  are split in $E$, then  $\mathcal{M}_{K,\nu}^{ss}$  is empty. \\

\textbf{\underline{$\boldsymbol{m=2}:$} }\\
Define the following constant based on the prime $\nu$:
$$c_\nu \  =  \ \text{ the number of } \  (a_{\nu,i}, b_{\nu,i}) = (0,2)  \ \text{ such that } \ \mathfrak{p_i} \ \text{ is split in } \ E.$$
If $c_\nu= 0$, the reduced $k$-scheme $\mathcal{M}_{K,\nu}^{ss}$ is zero-dimensional.  Otherwise, $\mathcal{M}_{K,\nu}^{ss}$ is empty. \\

\textbf{\underline{$\boldsymbol{m=3}:$} }\\
If any of the primes  $\mathfrak{p}_i$ are split in $E$, the supersingular locus  $\mathcal{M}_{K,\nu}^{ss}$ is empty. Otherwise, define the following constant which depends only on the global signature condition:
$$c \ =  \ \text{ the number of } \  (a_{i}, b_{i}) = (1,2) $$
Then the reduced $k$-scheme $\mathcal{M}_{K,\nu}^{ss}$  has geometry as follows:
\begin{itemize}
\item{The supersingular locus $\mathcal{M}_{K,\nu}^{ss}$ is equi-dimensional of dimension $c$.}
\item{Each irreducible component of  $\mathcal{M}_{K,\nu}^{ss}$ is isomorphic over $k$ to $C^{c}$,
 a product (over $k$) of $c$ copies of the Fermat curve $C$.}
 \item{Any two irreducible components either intersect trivially or have intersection isomorphic over $k$ to $C^{t}$
 for some integer $0 \leq t \leq c$. \\}
\end{itemize}

\textbf{\underline{$\boldsymbol{m=4}:$} }\\
Define the following constants based on the prime $\nu$:
$$c_\nu \ =  \ \text{ the number of } \  (a_{\nu,i}, b_{\nu,i}) = (0,4)  \ \text{ such that } \ \mathfrak{p_i} \ \text{ is inert in } \ E$$
$$d_\nu \ =  \ \text{ the number of } \   (a_{\nu,i}, b_{\nu,i})  = (1,3)  \ \text{ such that } \ \mathfrak{p_i} \ \text{ is inert in } \ E$$
$$e_\nu \ =  \ \text{ the number of } \   (a_{\nu,i}, b_{\nu,i})  = (2,2)  \ \text{ such that } \ \mathfrak{p_i} \ \text{ is inert in } \ E$$
$$f_\nu  \ =  \ \text{ the number of } \   (a_{\nu,i}, b_{\nu,i})  = (2,2)  \ \text{ such that } \ \mathfrak{p_i} \ \text{ is split in } \ E$$
$$g_\nu \ = \ \text{ the number of } \   (a_{\nu,i}, b_{\nu,i}) \neq (2,2)  \ \text{ such that } \ \mathfrak{p_i} \ \text{ is split in } \ E.$$
Note that any of these integers could be zero, and that $c_\nu + d_\nu + e_\nu + f_\nu + g_\nu = n$. Either $g_\nu > 0$, in which case $\mathcal{M}_{K,\nu}^{ss}$ is empty, or $g_\nu = 0$, in which case $\mathcal{M}_{K,\nu}^{ss}$ has geometry as follows:

\begin{itemize}
\item{The supersingular locus   $\mathcal{M}_{K, \nu}^{ss}$ is equi-dimensional  of dimension $d_\nu+2e_\nu+f_\nu$.\\}

\item{Each irreducible component of $\mathcal{M}_{K, \nu}^{ss}$ is isomorphic over $k$ to:
$$ C^{d_\nu} \times_k  S^{e_\nu} \times_k (\mathbb{P}^1)^{f_\nu}$$ 
 a product (over $k$) of $d_\nu$ copies of the Fermat curve $C$, of $e_\nu$ copies of the Fermat surface $S$, and of $f_\nu$ copies of the projective line. \\}
 
\item{ Any two irreducible components either intersect trivially or have intersection isomorphic over $k$ to:
 $$ C^r \times_k S^{s_1} \times_k (\mathbb{P}^1)^{s_2 +t}$$
 for some nonnegative integers $r, s_1, s_2, t$ satisfying the following three conditions:
\begin{enumerate}
\item{ $$0 \leq r \leq d_\nu$$ }
\item{ $$0 \leq s_1 + s_2 \leq e_\nu$$ }
\item{ $$0 \leq t \leq f_\nu.$$ }
\end{enumerate} }
\end{itemize}
\end{theorem}

	\bibliographystyle{plain}

\bibliography{testbib4}

\begin{thebibliography}{1}

\bibitem{formal}
Leovigildo {Alonso}, Ana {Jerem\'{\i}as}, and Marta {P\'{e}rez}.
\newblock {Infinitesimal local study of formal schemes}.
\newblock {\em arXiv Mathematics e-prints}, page math/0504256, Apr 2005.

\bibitem{GL4}
Maria {Fox}.
\newblock {The $\mathrm{GL}_4$ {R}apoport-{Z}ink Space}.
\newblock {\em Int. Math. Res. Not. IMRN}, to appear.

\bibitem{GU22}
Benjamin Howard and Georgios Pappas.
\newblock On the supersingular locus of the {${\rm GU}(2,2)$} {S}himura
  variety.
\newblock {\em Algebra Number Theory}, 8(7):1659--1699, 2014.

\bibitem{Ko}
Robert~E. Kottwitz.
\newblock Points on some {S}himura varieties over finite fields.
\newblock {\em J. Amer. Math. Soc.}, 5(2):373--444, 1992.

\bibitem{local}
Stephen Kudla and Michael Rapoport.
\newblock Special cycles on unitary {S}himura varieties {I}. {U}nramified local
  theory.
\newblock {\em Invent. Math.}, 184(3):629--682, 2011.

\bibitem{RZ}
M.~Rapoport and Th. Zink.
\newblock {\em Period spaces for {$p$}-divisible groups}, volume 141 of {\em
  Annals of Mathematics Studies}.
\newblock Princeton University Press, Princeton, NJ, 1996.

\bibitem{V}
Inken Vollaard.
\newblock The supersingular locus of the {S}himura variety for {${\rm
  GU}(1,s)$}.
\newblock {\em Canad. J. Math.}, 62(3):668--720, 2010.

\bibitem{VW}
Inken Vollaard and Torsten Wedhorn.
\newblock The supersingular locus of the {S}himura variety of {${\rm
  GU}(1,n-1)$} {II}.
\newblock {\em Invent. Math.}, 184(3):591--627, 2011.

\end{thebibliography}

\end{document}